\documentclass[12pt]{amsart}
\usepackage[all]{xypic}
\usepackage{tikz}
\usetikzlibrary{arrows}
\usetikzlibrary{decorations.markings}

\usepackage{graphicx}
\usepackage{bm}
\usepackage{epsf}
\usepackage{verbatim} 
\usepackage{amsmath}
\usepackage{amsfonts}
\usepackage{amssymb}
\usepackage{mathrsfs}
\usepackage{amsthm}
\usepackage{newlfont}

 \newtheorem{thm}{Theorem}
 
 \newtheorem{cor}[thm]{Corollary}

 \newtheorem{lem}[thm]{Lemma}
 
 \newtheorem{prop}[thm]{Proposition}
 
 \theoremstyle{definition}
 \newtheorem{defn}[thm]{Definition}
 
 \theoremstyle{definition}
 \newtheorem{notn}[thm]{Notation}
 
  \theoremstyle{definition}

 \theoremstyle{remark}
 \newtheorem{rem}[thm]{Remark}
 
 \theoremstyle{remark}

 \theoremstyle{definition}
 \newtheorem{example}[thm]{Example}

\numberwithin{thm}{section}
\numberwithin{equation}{section}

 \newcommand{\id}{\mathrm{id}}
 
 \newcommand{\ab}{\mathrm{ab}}
 
 \newcommand{\Hom}{\mathrm{Hom}}
 
 \newcommand{\Fix}{\mathrm{Fix}}
 \newcommand{\Spec}{\mathrm{Spec}}
 \newcommand{\Frob}{\mathrm{Frob}}

 \newcommand{\Aut}{\mathrm{Aut}}
 
 \newcommand{\End}{\mathrm{End}}
 \newcommand{\Pic}{\mathrm{Pic}}
 
 \newcommand{\ord}{\mathrm{ord}}

 \newcommand{\Gal}{\mathrm{Gal}}
 \newcommand{\GL}{\mathrm{GL}}

 \newcommand{\sep}{\mathrm{sep}}
 \newcommand{\alg}{\mathrm{alg}}
 \newcommand{\coker}{\mathrm{coker}}
 
 \newcommand{\rank}{\mathrm{rank}}

 \newcommand{\opp}{\mathrm{opp}}

\newcommand{\inv}{\mathrm{inv}}
 
 \newcommand{\Lie}{\mathrm{Lie}}
 \newcommand{\Ram}{\mathrm{Ram}}
 \newcommand{\chr}{\mathrm{char}}
 \renewcommand{\mod}{\mathrm{mod}}
 
 \newcommand{\diag}{\mathrm{diag}}

 \newcommand{\fp}{\mathfrak p}
 
 \newcommand{\fr}{\mathfrak r}
 \newcommand{\fn}{\mathfrak n}
 \newcommand{\fm}{\mathfrak m}
 
 \newcommand{\fc}{\mathfrak c}

 \newcommand{\fP}{\mathfrak P}

 \newcommand{\cO}{\mathcal{O}}
 \newcommand{\cV}{\mathcal{V}}
 
 \newcommand{\cM}{\mathcal{M}}
 \renewcommand{\cH}{\mathcal{H}}

 \newcommand{\cE}{\mathcal{E}}

 \newcommand{\cA}{\mathcal{A}}
 \newcommand{\cB}{\mathcal{B}}
 \newcommand{\cG}{\mathcal{G}}

 \newcommand{\cS}{\mathcal{S}}
 \renewcommand{\cL}{\mathcal{L}}

\newcommand{\sD}{\mathscr{D}}

 \newcommand{\gm}{\mathbb{G}}

 \newcommand{\C}{\mathbb{C}}
 \newcommand{\F}{\mathbb{F}}
 \newcommand{\Q}{\mathbb{Q}}
 
 \newcommand{\Z}{\mathbb{Z}}
 \newcommand{\E}{\mathbb{E}}
 \newcommand{\A}{\mathbb{A}}
 
 \newcommand{\p}{\mathbb{P}}
 \newcommand{\T}{\mathbb{T}}
 \newcommand{\M}{\mathbb{M}}
 
 \newcommand{\Nr}{\mathrm{Nr}}

 \newcommand{\G}{\Gamma}
 \newcommand{\To}{\longrightarrow}
 \newcommand{\bs}{\setminus}
 \newcommand{\Fi}{F_\infty}

 \newcommand{\La}{\Lambda}
 \newcommand{\la}{\lambda}

 \newcommand{\tF}{\widetilde{F}}

 \newcommand{\twist}[1]{{^\tau}\!#1}
  \newcommand{\DES}{\mathbf{DES}}
  \newcommand{\DMot}{\mathbf{DMot}}

\begin{document}

\title[Drinfeld-Stuhler modules]{Drinfeld-Stuhler modules}

\author{Mihran Papikian}
\address{Department of Mathematics, Pennsylvania State University, University Park, PA 16802, U.S.A.}
\email{papikian@psu.edu}

\thanks{The author's research was partially supported by grants from the Simons Foundation (245676) and the National Security Agency 
(H98230-15-1-0008).} 




\begin{abstract} We study $\sD$-elliptic sheaves in terms 
of their associated modules, which we call Drinfeld-Stuhler modules. First, we 
prove some basic results about Drinfeld-Stuhler modules and give explicit examples. Then we 
examine the existence and properties of Drinfeld-Stuhler modules 
with large endomorphism rings, which are analogous to CM and supersingular Drinfeld modules. 
Finally, we examine the fields of moduli of Drinfeld-Stuhler modules. 
\end{abstract}

\maketitle


\section{Introduction} The idea of $\sD$-elliptic sheaves was proposed by Ulrich Stuhler, 
as a natural generalization of Drinfeld's elliptic sheaves \cite{DrinfeldES}, \cite{BS}. The moduli 
varieties of $\sD$-elliptic sheaves were studied by Laumon, Rapoport and Stuhler in \cite{LRS}, 
with the aim of proving the local Langlands correspondence for $\GL_d$ in positive characteristic.  
In this paper, we study some of the arithmetic properties of $\sD$-elliptic sheaves, and in particular
their endomorphism rings and fields of moduli. 

Let $C$ be a smooth, projective, geometrically connected curve over the finite field $\F_q$. 
Let $F$ be the function field of $C$. 
Let $\infty\in C$ be a fixed closed point, and $A\subset F$ be the ring of functions regular outside $\infty$. 
Denote by $\Fi$ the completion of $F$ at $\infty$, and by $\C_\infty$ the completion of an algebraic closure of $\Fi$. 
Let $D$ be a central simple algebra over $F$ of dimension $d^2$, which is split at $\infty$, i.e., 
$D\otimes_F F_\infty$ is isomorphic to the matrix algebra $ M_d(F_\infty)$. 
Fix a maximal $A$-order $O_D$ in $D$. 
An \textit{$A$-field} is a field $L$ equipped with an $A$-algebra structure, i.e., with a homomorphism $\gamma: A\to L$. 
A $\sD$-elliptic sheaf over an $A$-field $L$ is essentially a vector bundle of rank $d^2$ on $C\times_{\Spec(\F_q)}\Spec(L)$  
equipped with an action of $O_D$ and with a meromorphic $O_D$-linear Frobenius 
satisfying certain conditions (see Section \ref{sMotSht}). One can think of these objects 
as being analogous to abelian varieties equipped with an action of an order 
in a central simple algebra over $\Q$. 

In this paper, we study $\sD$-elliptic sheaves in terms of their associated modules, which we call 
\textit{Drinfeld-Stuhler modules}. The relationship between $\sD$-elliptic sheaves and 
Drinfeld-Stuhler modules is similar to the relationship between elliptic sheaves and Drinfeld modules; cf. \cite{DrinfeldES}, \cite{BS}. 
Let $L$ be an $A$-field, and 
$L[\tau]$ be the skew polynomial ring with the commutation relation $\tau b=b^q\tau$, $b\in L$. A 
Drinfeld-Stuhler $O_D$-module over $L$ is an embedding 
$$
\phi: O_D\To M_d(L[\tau])
$$
satisfying certain conditions; see $\S$\ref{ssD&BP}. 
This concept implicitly appears in \cite[$\S$3]{LRS}, although it 
does not play an important role in that paper since its ``shtuka'' incarnation (the $\sD$-elliptic sheaf) seems  
better suited for the study of moduli spaces. 
The advantage of the concept of Drinfeld-Stuhler module is that it is relatively elementary, 
and one can easily write down explicit examples of these objects. 
We expect that the reader familiar with the theory of Drinfeld modules, but not necessarily 
with \cite{LRS}, will find it easier to understand the results of this paper in terms of 
Drinfeld-Stuhler modules, rather than $\sD$-elliptic sheaves. 

The theory of Drinfeld modules can be seen as a 
special case of the theory of Drinfeld-Stuhler modules for $D\cong M_d(F)$; cf. $\S$\ref{sME}. 
In the other direction, some of the properties of general Drinfeld-Stuhler modules are similar to, and in fact can be deduced from, the 
properties of Drinfeld modules, e.g., uniformizability and CM theory. 
There are also some notable 
differences. The most significant is probably the fact that, when $D$ is a division algebra, the modular varieties of  Drinfeld-Stuhler modules 
are projective \cite{LRS}, unlike the Drinfeld modular varieties, which are affine \cite{Drinfeld}. Another difference is that 
Drinfeld-Stuhler modules can be defined only over fields which split $D$ (cf. Lemma \ref{lemFD}), so for $D\not\cong M_d(F)$ 
there are no Drinfeld-Stuhler modules over $F$ itself, even in the simplest case when $A=\F_q[T]$. 

The main results of this paper 
concern the endomorphism rings of  
Drinfeld-Stuhler modules and their fields of moduli. The outline of the paper is the following: 

In Section \ref{sBasics}, we introduce the concept of Drinfeld-Stuhler $O_D$-modules and prove some of its basic properties. 
Moreover, we give several explicit examples, which will be revisited throughout the paper. Finally, we explain the 
so-called Morita equivalence in the context of Drinfeld-Stuhler modules, which gives 
an equivalence between the categories of Drinfeld $A$-modules of rank $d$ 
and Drinfeld-Stuhler $M_d(A)$-modules. 

In Section \ref{sMotSht}, we recall three categories equivalent to the category of Drinfeld-Stuhler modules: $O_D$-motives, $\sD$-elliptic sheaves, 
and $O_D$-lattices.  The tools provided by these alternative points of view on Drinfeld-Stuhler modules are crucial for the 
proofs of the main results of the paper. 

In Section \ref{sCM}, given a Drinfeld-Stuhler $O_D$-module $\phi$ over $L$, we prove that  the endomorphism ring 
$\End_L(\phi)$(= the centralizer of $\phi(O_D)$ in $M_d(L[\tau])$) is a projective $A$-module 
of rank $\leq d^2$ such that $\End_L(\phi)\otimes_A \Fi$ is isomorphic to a subalgebra of the central 
division algebra over $\Fi$ with invariant $-1/d$. 
Moreover, if $\gamma: A\to L$ is injective, then $\End_L(\phi)$ is an $A$-order 
in an \textit{imaginary} field extension $K$ of $F$ which embeds into $D$, so, in particular, 
$\End_L(\phi)$ is commutative and its rank over $A$ divides $d$. 
 (``Imaginary'' in this context means 
that there is a unique place $\infty'$ of $K$ over $\infty$.) 
Next, we study Drinfeld-Stuhler modules over $\C_\infty$ with large endomorphism rings, namely the 
analogue of ``complex multiplication''. The results here are similar 
to those for Drinfeld modules; cf. \cite{GekelerADM}, \cite{HayesCFT}. We prove 
that if $K$ is an imaginary field extension of $F$ of degree $d$ which embeds into $D$ and $O_K$ is the integral closure of $A$ in $K$, 
then, up to isomorphism, the number of Drinfeld-Stuhler $O_D$-modules over $\overline{F}$
with $\End_{\overline{F}}(\phi)=O_K$ is finite and non-zero, and any such module 
can be defined over the Hilbert class field of $K$(= the maximal unramified abelian extension of $K$ in which $\infty'$ totally splits).
 We also compute the number of isomorphism classes of Drinfeld-Stuhler modules over $\C_\infty$
with the largest possible automorphism group $\F_{q^d}^\times$, assuming $A=\F_q[T]$. 

In Section \ref{sSS}, we fix a maximal ideal $\fp\lhd A$ and study Drinfeld-Stuhler $O_D$-modules over the algebraic closure 
of $A/\fp$ with large endomorphism rings, namely the 
analogue of ``supersingularity''. It turns out that the cases when $D$ is ramified/unramified at $\fp$ have to be studied separately, but 
in either case the endomorphism ring of a supersingular Drinfeld-Stuhler module 
is essentially a maximal $A$-order in the central division algebra over $F$ of dimension $d^2$ 
whose invariants are closely related to the invariants of $D$. These results are again similar 
to those for Drinfeld modules; cf. \cite{GekelerFDM}.  
This potentially opens up the way to use Drinfeld-Stuhler modules in the arithmetic 
of division algebras over function fields. 

In Section \ref{sFM}, we prove a Hilbert's 90-th type theorem for $M_d(L^\sep[\tau])$, and use this 
theorem to give conditions under which a field of moduli for a Drinfeld-Stuhler module is a field of definition. 
In particular, we prove that a field of moduli is a field of definition if and only if 
it splits $D$. We also prove that if $d$ and $q^d-1$ are coprime, then a field of moduli is always a field of definition. 
These results have applications to the existence/non-existence of rational points on 
the coarse moduli scheme of Drinfeld-Stuhler modules. An interesting application of this is the construction of concrete 
examples of varieties over function fields violating the Hasse principle; we will discuss this application in a future publication.   


\section{Basic properties and examples}\label{sBasics}

In this section, after introducing the notation and terminology that will be used throughout the paper, we define the 
key concept of Drinfeld-Stuhler module. We then examine the basic properties of these objects and give explicit examples. 

\subsection{Notation and terminology}
Let $F$ be the field of rational functions on a smooth and geometrically irreducible projective curve $C$ defined over 
the finite field $\F_q$ of $q$ elements, where $q$ is a power of a prime number. 
Fix a place $\infty$ of $F$ (equiv. a closed point of $C$), and let $A$ be the subring of $F$ consisting of functions which are regular away from $\infty$.  
$A$ is a Dedekind domain. 
 
 An \textit{imaginary} field extension of $F$ is an extension $K/F$ in which $\infty$ does not split, i.e., there is a unique place $\infty'$
 of $K$ over $\infty$.    
For a field $L$ we denote by $L^\alg$ (resp. $L^\sep$) its algebraic (resp. separable) closure. 

For a place $v$ of $F$, we denote by $F_v$, $O_v$, $\F_v$ the completion of $F$ at $v$, the ring of integers in $F_v$, 
and the residue field at $v$, respectively. If $v\neq \infty$, so corresponds to a non-zero prime ideal $\fp$ of $A$, 
we sometimes write $A_\fp$ or $A_v$ instead of $O_v$, and $\F_\fp$ instead of $\F_v$. 
The maximal ideal of $O_v$ will be denoted $\fp$ if $v=\fp$ is finite, and $\fp_\infty$ if $v=\infty$.  

Given a unitary ring $R$, we denote by $R^\times$ the group of multiplicative units in $R$. Let 
$M_d(R)$ be the ring of $d\times d$ matrices with entries in $R$; 
the group of units in $M_d(R)$ is denoted by $\GL_d(R)$. Given $r_1, \dots, r_d\in R$, we denote by 
$\diag(r_1, \cdots, r_d)\in M_d(R)$ the matrix which has $r_i$ as the $(i, i)$-th entry, $1\leq i\leq d$, 
and zeros everywhere else. 

Let $D$ be a central simple algebra over $F$ of dimension $d^2$. 
Let $\Ram(D)$ be the set of places of $F$ which ramify in $D$, i.e., $v\in \Ram(D)$ 
if and only if $D_v:=D\otimes_F F_v$ is not isomorphic to $M_d(F_v)$. From now on we assume 
that $\infty\not \in \Ram(D)$, so that the places in $\Ram(D)$ correspond to prime ideals of $A$. We denote 
$$
\fr(D)=\prod_{\fp\in \Ram(D)} \fp. 
$$
An empty product is assumed to be $1$, so $\fr(M_d(F))=A$. 
We fix a maximal $A$-order $O_D$ in $D$; see \cite{Reiner} for the definitions. Note that $A$ is the center of $O_D$. 

Let $L$ be an $A$-field, i.e., a field equipped with an $A$-algebra structure $\gamma:A\to L$. 
The \textit{$A$-characteristic of $L$} is the prime ideal $\chr_A(L):=\ker(\gamma)\lhd A$; we say that $L$ 
has \textit{generic $A$-characteristic} if $\ker(\gamma)=0$. 
Note that $\F_q$ is a subfield of $L$.  Let $\tau$ be the $\F_q$-linear Frobenius endomorphism of the 
additive group-scheme $\gm_{a,L}=\Spec(L[x])$ over $L$; the morphism $\tau$ is given on the underlying ring by $x\mapsto x^q$. 
The ring of $\F_q$-linear endomorphisms $\End_{\F_q}(\gm_{a, L})$ is canonically isomorphic to the 
skew polynomial ring $L[\tau]$ with the commutation relation $\tau b=b^q\tau$, $b\in L$. 
There is also a canonical isomorphism (cf. \cite[Prop. 1.1]{vdHeiden})
$$
\End_{\F_q}\left(\gm_{a,L}^d\right)\cong M_d\left(L[\tau]\right). 
$$
One can write the elements of $M_d(L[\tau])$ as finite sums $\sum_{i\geq 0} B_i\tau^i$, where $B_i\in M_d(L)$ 
and, by slight abuse of notation, $\tau^i$ denotes $\diag(\tau^i, \dots, \tau^i)$. 
An element $S=\sum_{i\geq 0} B_i\tau^i\in M_d(L[\tau])$ acts on the tangent space $\Lie(\gm_{a, L}^d)\cong L^d$ 
via $\partial(S):=B_0$. The map
$$\partial: M_d(L[\tau])\To M_d(L), \quad S\longmapsto \partial(S)$$ is a surjective homomorphism. 

\subsection{Definitions and basic properties}\label{ssD&BP}
\begin{defn}\label{defDSM}
A \textit{Drinfeld-Stuhler $O_D$-module} defined over $L$ is an embedding 
\begin{align*}
\phi: O_D &\To M_d(L[\tau]) \\ 
b &\longmapsto \phi_b
\end{align*}
satisfying the following conditions: 
\begin{itemize}
\item[(i)] For any $b\in O_D\cap D^\times$ the endomorphism $\phi_b$ of $\gm_{a, L}^d$ 
is surjective with kernel $\phi[b]:=\ker \phi_b$ a finite group scheme over $L$ of order $\# (O_D/O_D\cdot b)$. 
\item[(ii)] The composition 
$$
A\To O_D\overset{\phi}{\To} M_d(L[\tau]) \overset{\partial}{\To} M_d(L)
$$
maps $a\in A$ to $\diag(\gamma(a), \dots, \gamma(a))$, where the map $A\to O_D$ identifies $A$ with the center of $O_D$.  
\end{itemize}
\end{defn}

\begin{rem}
	A homomorphism $f: \gm_{a, L}^d \to \gm_{a, L}^d$ is surjective if and only if $\ker f$ is finite; cf. \cite[Prop. 5.2]{HartlIsog}. In particular,  
	(i) in Definition \ref{defDSM} can be simplified to the assumption that $\#\phi[b]=\# (O_D/O_D\cdot b)$ for all $b\in O_D\cap D^\times$. 
\end{rem}

\begin{rem} Drinfeld-Stuhler modules are a special case of abelian Anderson $A$-modules, as follows from \cite[$\S$1]{Anderson} and 
	the discussion in Section \ref{sMotSht}. 
	If $d=1$, so that $D=F$, then the definition of Drinfeld-Stuhler $O_D$-modules becomes the definition of Drinfeld $A$-modules of rank $1$. 
	One can introduce a notion of rank for Drinfeld-Stuhler modules so that Definition \ref{defDSM} corresponds to the 
	case of rank $1$. Lafforgue studied the $\sD$-shtukas of arbitrary rank and their modular varieties in \cite{Lafforgue}. 
\end{rem}	
\begin{rem}Let $b\in O_D\cap D^\times$. If we consider $D$ as a vector space over $F$, then the left multiplication by $b$ 
	induces a linear transformation. Let $\det(b)$ denote the determinant of this linear transformation. 
	Note that $O_D$ is an $A$-lattice in $D$ in the sense of \cite[Ch. III, $\S$1]{SerreLF}. 
	By Proposition 3 in \cite[Ch. III, $\S$1]{SerreLF}, we have $\# (O_D/O_Db)=\# (A/\det(b)A)$. 
	Finally, recall that $(-1)^d \det(b)=: \Nr(b)$ is the \textit{non-reduced norm} of $b$; cf. \cite[$\S$9a]{Reiner}.
	Hence condition (i) is equivalent to saying that $\phi[b]$ is a finite group scheme of order $\#(A/\Nr(b)A)$. 
\end{rem}

The action of $\phi(O_D)$ on the tangent space $\Lie(\gm_{a, L}^d)$ gives a 
homomorphism $$\partial_\phi: O_D\To M_d(L),$$ which extends linearly 
to a homomorphism 
$$
\partial_{\phi, L}: O_D\otimes_A L \To M_d(L). 
$$
\begin{lem}\label{lemFD} If $\chr_A(L)$ does not divide $\fr(D)$, then 
$\partial_{\phi, L}$ is an isomorphism.  
\end{lem}
\begin{proof}
Both sides are rings with $1$, so $\partial_{\phi, L}$ is non-zero, as it maps $1$ to $1$. 
If $L$ has generic $A$-characteristic, then $L$ is an extension of $F$, hence 
$O_D\otimes_A L$ is a central simple algebra over $L$. Therefore, $\partial_{\phi, L}$ is injective, and comparing the dimensions we see 
that it is in fact an isomorphism. 
Now assume that $\chr_A(L)=\fp\neq 0$. 
Then $O_D\otimes_A L$ is obtained by extension of scalars from $O_D\otimes_A A_\fp\to O_D\otimes_A \F_\fp$. 
On the other hand, $O_D\otimes_A A_\fp\cong M_d(A_\fp)$ since $\fp\nmid \fr(D)$. 
Now it is clear that $O_D\otimes_A \F_\fp\cong M_d(\F_\fp)$, hence $O_D\otimes_A L\cong M_d(\F_\fp)\otimes_{\F_\fp}L\cong M_d(L)$. 
Since $M_d(L)$ is a central simple algebra over $L$, the previous argument again implies that $\partial_{\phi, L}$ is an isomorphism. 
\end{proof}

\begin{rem}\label{remNotMd}
	If $\chr_A(L)$ divides $\fr(D)$, then 
	$\partial_{\phi, L}$ is not an isomorphism, since  $O_D\otimes_A L$ is not isomorphic to $M_d(L)$; cf. $\S$\ref{ssSScase2}.  
\end{rem}

\begin{rem}\label{remSplitD}
We recall some necessary and sufficient conditions for a finite field extension $L$ of $F$ 
to \textit{split} $D$, i.e., $D\otimes_F L\cong M_d(L)$, 
since, by Lemma \ref{lemFD}, if there is a Drinfeld-Stuhler $O_D$-module defined over $L$, then $L$ necessarily splits $D$. 
(In particular, if $D\not\cong M_d(F)$, then a Drinfeld-Stuhler module cannot be defined over $F$ itself.)
Let $\fp\lhd A$. The Wedderburn structure theorem says that 
$D\otimes_{F}F_\fp\cong M_{\kappa_\fp}(D_\fp')$, where $D_\fp'$ is a central division algebra 
of dimension $d_\fp^2=(d/\kappa_\fp)^2$. The integer $d_\fp$ is called 
the \textit{local index} of $D$ at $\fp$; cf. \cite[p. 272]{Reiner}. 
 By \cite[(32.15)]{Reiner}, $L$ splits $D$ if and only if for each prime $\fp\lhd A$ 
and for all primes $\fP$ of $L$ lying above $\fp$, $d_\fp$ divides $[L_\fP:K_\fp]$. If $D$ is a division algebra 
and $[L:F]=d$, then $L$ splits $D$ if and only if $L$ embeds into $D$; moreover, every maximal subfield $L$ of $D$ contains $F$ 
and $[L:F]=d$; see \cite[(7.15)]{Reiner}.   
\end{rem}



\begin{defn}\label{defn2.6Hom}
 Let $\phi, \psi$ be Drinfeld-Stuhler $O_D$-modules over $L$. 
A \textit{morphism} $u:\phi\to \psi$ over $L$ is 
$u\in M_d(L[\tau])$ such that $u\phi_b=\psi_b u$ for all $b\in O_D$. 
We say that $u$ is an \textit{isomorphism} if $u$ is invertible in the ring $M_d(L[\tau])$. 
We say that $u$ is an \textit{isogeny} if $\ker(u)$ is a finite group scheme over $L$; 
an isogeny $u$ is \textit{separable} if $\ker(u)$ is \'etale. Note that $\phi_b$, $b\in O_D\cap D^\times$, 
defines an isogeny $\phi\to \phi$. 
The set of morphisms $\phi\to \psi$ over $L$ is an $A$-module $\Hom_L(\phi, \psi)$, where 
$A$ acts by $a\circ u:=u\phi_a$. (Using the fact that $a\in A$ is in the center of $O_D$, it is 
easy to check that $u\phi_a\in \Hom_L(\phi, \psi)$.) 
We denote $\End_L(\phi)=\Hom_L(\phi, \phi)$; this is a subring of $M_d(L[\tau])$. 
For an arbitrary field extension $\cL$ of $L$ we can consider $\phi, \psi$ as Drinfeld-Stuhler $O_D$-modules over $\cL$, 
so we have the corresponding module $\Hom_\cL(\phi, \psi)$ of morphisms over $\cL$. We will denote 
$\Hom(\phi, \psi)=\Hom_{L^\alg}(\phi, \psi)$ and $\End(\phi)=\End_{L^\alg}(\phi)$.  

	A Drinfeld-Stuhler $O_D$-module $\phi$ over $L$ \textit{can be defined} over a subfield $K$ of $L$ 
	(equiv. $K$ is a \textit{field of definition for $\phi$}) 
	if there is a Drinfeld-Stuhler $O_D$-module $\psi$ over $K$ which is isomorphic to $\phi$ over $L$. 
	\end{defn}

\begin{lem}\label{lemSepIsog}
	Let $\phi$ be a Drinfeld-Stuhler $O_D$-module over $L$ and $b\in O_D\cap D^\times$. Then $\phi_b$ is separable 
	if and only if $\chr_A(L)$ does not divide $\Nr(b)$.  
\end{lem}
\begin{proof} 
	This follows from Proposition 3.10 in \cite{LRS} (see also Corollary 5.11 in \cite{HartlIsog}). 
\end{proof}

\begin{lem}\label{lem2.10}
	For a non-zero ideal $\fn\lhd A$ and a Drinfeld-Stuhler $O_D$-module $\phi$ over $L$ we define 
	$$
	\phi[\fn]=\bigcap_{a\in \fn}\phi[a],
	$$
	where the intersection is the scheme-theoretic intersection of subgroup schemes of $\gm_{a, L}^d$. 
	Then $\phi[\fn]$ is invariant under  $\phi(O_D)$. Moreover, if $\chr_A(L)$ does not divide $\fn$, 
	then $$\phi[\fn](L^\sep)\cong O_D/O_D\fn$$ as left $O_D$-modules.  
\end{lem}
\begin{proof}
	Since $a\in A$ is in the center of $O_D$, it is clear that each $\phi[a]$, and thus also $\phi[\fn]$, is $\phi(O_D)$-invariant. 
	The second claim essentially follows from Theorem 6.4 in \cite{HartlIsog}. More precisely, in the terminology 
	of Section \ref{sMotSht}, let $M(\phi)$ be the $O_D$-motive associated to $\phi$. By \cite[Thm. 6.4]{HartlIsog}, $\phi[\fn]$ 
	is dual to $M(\phi)/\fn M(\phi)$. This isomorphism is compatible with the action of $O_D$. 
	On the other hand, $M(\phi)$ is a locally free left $O_D^\opp\otimes_{\F_q} L$-module of rank $1$. Hence
	$M(\phi)/\fn M(\phi)\cong O_D^\opp/\fn O_D^\opp$ as left $O_D^\opp$-modules. 
\end{proof}

\begin{rem} In general, $\phi[b]$ is not necessarily $O_D$-invariant for $b\in O_D\cap D^\times$, 
	so condition (i) 
	in Definition \ref{defDSM} cannot be stated in the stronger form 
	of isomorphism of left $O_D$-modules $\phi[b]\cong O_D/O_D\cdot b$; see Remark \ref{remNotO_Disom}. 
\end{rem}

\begin{lem}\label{lemECom} Let $\phi$ and $\psi$ be Drinfeld-Stuhler $O_D$-modules over $L$. Assume $L$ has generic 
	$A$-characteristic. Then:
	\begin{enumerate}
		\item The map $\partial: \Hom_L(\phi, \psi)\to M_d(L)$ is injective.
		\item $\End_L(\phi)$ is a commutative ring. 
	\end{enumerate}
\end{lem}
\begin{proof} 
	Suppose $u\in \Hom_L(\phi, \psi)$ is non-zero but $\partial(u)=0$. Then $u=B_m\tau^m+B_{m+1}\tau^{m+1}+\cdots$, 
	where $m\geq 1$ is the smallest index such that $B_m\neq 0$. For $a\in A$, the equality $u\phi_a=\psi_a u$ 
	leads to $B_m\gamma(a)^{q^m}=\gamma(a)B_m$. Since $B_m\in M_d(L)$ has at least one non-zero entry, we must have 
	$\gamma(a)^{q^m}=\gamma(a)$. Since $a$ was arbitrary, this implies $\gamma(A)\subseteq \F_{q^m}$. On the 
	other hand, since $L$ has generic $A$-characteristic, $\gamma(A)$ is infinite, which leads to a contradiction.  
	
	By the first claim, $\partial$ maps $\End_L(\phi)$ isomorphically to its image in $M_d(L)$. On the other hand, 
	$\partial(\End_L(\phi))$ is in the centralizer of $\partial_\phi(O_D)$. By Lemma \ref{lemFD}, $\partial_\phi(O_D)$ 
	contains a basis of $M_d(L)$, so $\partial(\End_L(\phi))$ is in the center of $M_d(L)$, which consists of 
	scalar matrices. Hence $\partial$ identifies $\End_L(\phi)$ with an $A$-subalgebra of $L$. 
\end{proof}

\begin{lem}\label{lem1.7} 
	Let $\phi$ and $\psi$ be Drinfeld-Stuhler $O_D$-modules over $L$. 
If $u\in \Hom_L(\phi, \psi)$ is non-zero, then $u$ is an isogeny. 
\end{lem}
\begin{proof}
We will prove the lemma assuming $\chr_A(L)\nmid \fr(D)$. At the end of $\S$\ref{sME} we give a different 
proof, which avoids this assumption. 

Without loss of generality, we can assume that $L$ is algebraically closed. 
Suppose $u\in \Hom(\phi, \psi)$ is non-zero and has infinite kernel. Since $\ker(u)\subset \gm_{a, L}^d$ 
is an algebraic subgroup with infinitely many geometric points, the connected component $\ker(u)^0$ of the 
identity has positive dimension. 
We can decompose $u=u_0\tau^s$ for some $s\geq 0$, so that 
$\partial(u_0)\neq 0$. Note that $\partial_{u_0}$ is not invertible since it acts as $0$ on the tangent space of $\ker(u)^0$. Thus, 
$0\subsetneq \ker(\partial_{u_0})\subsetneq L^d$. 
Denote by $\partial_{\phi, L}^{q^s}$ the composition of $\partial_{\phi, L}$ and $\tau^s: M_d(L)\to M_d(L)$, 
which raises the entries of a matrix to $q^s$th powers. 
By Lemma \ref{lemFD}, since $L$ is algebraically closed, we have $\partial_{\phi, L}^{q^s}(O_D\otimes L)=M_d(L)$. 
On the other hand, $\partial_{u_0}\partial_{\phi, L}^{q^s}(b)=\partial_{\psi, L}(b)\partial_{u_0}$ for all $b\in O_D$, which comes from 
$u\phi_b=\psi_bu$.  
This implies that the subspace $\ker(\partial_{u_0})$ of $L^d$ is invariant under $M_d(L)$, 
which leads to a contradiction.  
\end{proof}

\begin{lem}\label{lemComEnd} Let $\phi$ and $\psi$ be Drinfeld-Stuhler $O_D$-modules over $L$.
	If $u: \phi\to \psi$ is an isogeny, then there is an element $0\neq a\in A$ and an isogeny $w:\psi\to \phi$ such that 
	$wu=\phi_a$ and $uw=\psi_a$. 
\end{lem}
\begin{proof} This follows from Corollary 5.15 in \cite{HartlIsog}. We remark that 
	an isogeny $u: \phi\to \psi$ between abelian Anderson $A$-modules 
	is defined in \cite{HartlIsog} with an extra assumption that $u$, as an endomorphism of $\gm_{a,L}^d$, is  surjective. 
	On the other hand, the surjectivity of $u$ follows from the finiteness of $\ker(u)$ (see \cite[Prop. 5.2]{HartlIsog}), so Hartl's definition is 
equivalent to Definition \ref{defn2.6Hom}. 
\end{proof}

Lemmas \ref{lemSepIsog} and \ref{lemComEnd} imply that any isogeny $u:\phi\to \psi$ between 
Drinfeld-Stuhler $O_D$-modules over a field $L$ of generic $A$-characteristic is separable. 
In fact, this is true more generally for isogenies between abelian Anderson $A$-modules over $L$; see \cite[Cor. 5.17]{HartlIsog}. 

Lemmas \ref{lem1.7} and \ref{lemComEnd} imply that $\End_L(\phi)\otimes_A F$ is a division algebra over $F$. In 
	particular, if $L$ has generic $A$-characteristic, then $\End_L(\phi)\otimes_A F$ is a field, since 
	it is commutative by Lemma \ref{lemECom} (see Theorem \ref{thmEnd} for a more precise statement).

\subsection{Examples}\label{ssExamples}
As a consequence of the Grunwald-Wang theorem, every central simple 
$F$-algebra is cyclic; see \cite[(32.20)]{Reiner}. This means that there is a 
Galois extension $K/F$ with $\Gal(K/F)\cong \Z/d\Z$, a generator $\sigma$ of $\Gal(K/F)$, and 
$f\in F^\times$ such that  
\begin{equation}\label{eq-CAP}
D\cong (K/F, \sigma, f)=\bigoplus_{i=0}^{d-1} K z^i, \qquad z^d=f, \qquad z\cdot y=\sigma(y) z \quad \text{for } y\in K,
\end{equation}
where we identify $z^0$ with the identity element of $D$. Moreover, one can choose $f$ to be in $A$; cf. \cite[(30.4)]{Reiner}. 

Assume $K/F$ is imaginary and let $O_K$ be the integral closure of $A$ in $K$. Consider the $A$-order 
\begin{equation}\label{eqMax?}
O_D=\bigoplus_{i=0}^{d-1} O_K z^i  
\end{equation}
in $D$. 
This order is not necessarily maximal. It is not hard to compute that its discriminant is equal 
to $f^{d(d-1)}\mathrm{disc}(K/F)^d$; see \cite[Cor. 7]{BG}. 
For an $A$-order in $D$ to be maximal, it is necessary and 
sufficient for its discriminant to be equal to the discriminant of a maximal order. The discriminant of a 
maximal order in $D$ can be computed from the invariants of $D$; see \cite[(32.1)]{Reiner} and \cite[Prop. 25]{BG}.  
For $\fp\in \Ram(D)$, let the reduced fraction $s_\fp/r_\fp\in \Q/\Z$ be the invariant of $D$ at $\fp$. 
Set $r=\mathrm{lcm}(r_\fp\mid \fp\in  \Ram(D))$. Then a maximal order in $D$ has discriminant 
$\left(\prod_{\fp\in \Ram(D)} \fp^{r-\frac{r}{r_\fp}}\right)^r$. For example, if $d$ 
is prime, then the discriminant of a maximal order is equal to $\fr(D)^{d(d-1)}$. 
Comparing the discriminant of $O_D$ with the discriminant of a maximal order gives an explicit criterion 
for the order $O_D$ to be maximal; see \cite[Cor. 26]{BG}.  

\begin{example}\label{example9}
Assume the order $O_D$ in \eqref{eqMax?} is maximal. 
Let $\Phi: O_K\to L[\tau]$ be a Drinfeld $O_K$-module of rank $1$ defined over some field $L$. 
Observe that the restriction of $\Phi$ to $A$ defines a Drinfeld $A$-module of rank $d$ over $L$. 
Let 
$$
\phi: O_D\To M_d(L[\tau])
$$
be defined as follows: 
\begin{align*}
\phi_\alpha &=\diag(\Phi_\alpha, \Phi_{\sigma\alpha}, \dots, \Phi_{\sigma^{d-1}\alpha}), \quad \alpha\in O_K, 
\\
\phi_z &=\begin{pmatrix} 
0 & 1 & 0 & \cdots & 0 \\ 
0 & 0 & 1 & \cdots & 0 \\ 
 &  &  & \ddots & \\
 0 & 0 & 0 & \cdots & 1 \\
 \Phi_f & 0 & 0 & \cdots & 0
\end{pmatrix}. 
\end{align*}
Using the fact that $\Phi_\alpha\Phi_f=\Phi_f\Phi_\alpha$, it is easy to check that $\phi_z\phi_\alpha=\phi_{\sigma\alpha}\phi_z$ and $\phi_z^d=\phi_f$. Thus, $\phi$ 
is an embedding. Moreover, for $a\in A$, 
we have $\phi_a=\diag(\Phi_a, \dots, \Phi_a)$, which maps under $\partial$ to $\diag(\gamma(a), \dots, \gamma(a))$ 
by the definition of Drinfeld modules. Finally, 
$$
\#\phi[z]=\#\ker \Phi_f = \# (A/f A)^d= \# (A/f^dA)=\# (A/\Nr(z)A), 
$$
and 
$$
\#\phi[\alpha]=\#(O_K/O_K\alpha)^d= \# (A/\Nr(\alpha)A).  
$$
Thus, $\phi$ is a Drinfeld-Stuhler $O_D$-module. 
\end{example}

\begin{example}\label{example12} 
As a more explicit version of Example \ref{example9}, 
let $A=\F_q[T]$ and $F=\F_q(T)$. Let $\F_{q^d}$ denote the degree $d$ extension of $\F_q$. 
Let $K=\F_{q^d}(T)$, which is a cyclic imaginary 
extension as $\infty$ is inert in $K$. 
In this case, $O_K=\F_{q^d}[T]$ and the Galois group $\Gal(K/F)\cong \Gal(\F_{q^d}/\F_q)$ has a canonical generator $\sigma$ given by the 
Frobenius automorphism (i.e., $\sigma$ induces the $q$th power morphism on $\F_{q^d}$). 
Let $\fr\in A$ be a monic square-free polynomial with prime decomposition $\fr=\fp_1\cdots \fp_m$. 
Assume the degree of each prime $\fp_i$ is coprime to $d$. Let $D$ be the cyclic algebra $D=(K/F, \sigma, \fr)$. 
Then, by \cite[Thm. 4.12.4]{Goss}, for any prime $\fp\lhd A$ one has 
\begin{equation}\label{eqGossRosen}
\inv_\fp(D)=\frac{\ord_\fp(\fr)\deg(\fp)}{d} \in \Q/\Z. 
\end{equation}
Since the sum of the invariants of $D$ over all places of $F$ is $0$, if we assume that $\sum_{i=1}^m\deg(\fp_i)$ is divisible by $d$, then 
$D$ will be split at $\infty$ and will ramify only at the primes of $A$ dividing $\fr$. 

The order $O_D=\bigoplus_{i=0}^{d-1}O_K z^i$ is maximal in $D$, since its discriminant is equal to $\fr^{d(d-1)}$. 
Let $L$ be an $O_K$-field and $\gamma: A\to O_K\to L$ be the composition homomorphism. 
Let $\Phi: O_K \to L[\tau]$ 
be defined by $\Phi_T=\gamma(T)+\tau^d$; this is a rank-$1$ Drinfeld $O_K$-module and 
a rank-$d$ Drinfeld $A$-module. 
Then  
$$
\phi: O_D\To M_d(L[\tau]) 
$$
given by 
\begin{align*}
\phi_T &=\diag(\Phi_T, \dots, \Phi_T), \\
\phi_h &=\diag(h, h^q, \dots, h^{q^{d-1}}), \quad h\in \F_{q^d},\\
\phi_z &=\begin{pmatrix} 
0 & 1 & 0 & \cdots & 0 \\ 
0 & 0 & 1 & \cdots & 0 \\ 
 &  &  & \ddots & \\
 0 & 0 & 0 & \cdots & 1 \\
 \Phi_\fr & 0 & 0 & \cdots & 0
\end{pmatrix}, 
\end{align*}
is a Drinfeld-Stuhler module. 
\end{example}

\begin{rem}\label{remNotO_Disom}
It is easy to see from the previous example that for general $b\in O_D$ the kernel $\phi[b]$ is not necessarily $O_D$-invariant. 
Indeed, take $d=2$ and $b=h+z$ with $h\in \F_{q^2}\setminus \F_q$. A non-zero element  
$\begin{pmatrix} \alpha \\ \beta\end{pmatrix}\in \gm_{a, K}^2(\overline{K})$ is in $\phi[b]$ only if $h\alpha+\beta=0$. 
On the other hand, $\phi_h \begin{pmatrix} \alpha \\ \beta\end{pmatrix} = \begin{pmatrix} h\alpha \\ h^q\beta\end{pmatrix}$, so 
$\phi_h \begin{pmatrix} \alpha \\ \beta\end{pmatrix}\in \phi[b]$ only if $h^2\alpha+h^q\beta=0$. This implies 
$h^2\alpha=h^{q+1}\alpha$. Since $h^{q-1}\neq 1$, we must have $\alpha=0$, but then $\beta=0$.  
\end{rem}

\begin{example}\label{exampleMdA} Let $D=M_d(F)$ and $O_D=M_d(A)$. 
	Let $\Phi: A\to L[\tau]$ be a Drinfeld $A$-module over $L$ of rank $d$. Define 
	\begin{align*}
	\phi: O_D &\To M_d(L[\tau]) \\ 
	(a_{ij}) &\longmapsto \left(\Phi_{a_{ij}}\right).
	\end{align*}
	It is easy to check that $\phi$ is an injective homomorphism using the fact that $\Phi: A\to L[\tau]$ is an injective homomorphism.  
	That (ii) is satisfied follows from 
	the definition of Drinfeld modules. The non-reduced norm on $O_D$ in this case is simply the $d$th power of the determinant map, up to a sign. 
	Condition (i) is easy to check for diagonal and unipotent matrices in $M_d(A)$. Since these matrices generate the semigroup 
	of matrices in $M_d(A)$ with non-zero determinants, it follows that condition (i) holds. Hence $\phi$ 
	is a Drinfeld-Stuhler module. 
\end{example}

\subsection{Morita equivalence for Drinfeld-Stuhler modules}\label{sME}

The main result of this subsection is the fact that any Drinfeld-Stuhler $M_d(A)$-module arises from some Drinfeld module of rank $d$
via the construction of Example \ref{exampleMdA}. This fact for $\sD$-elliptic sheaves is mentioned in \cite[p. 224]{LRS}. 

Let $R$ be an arbitrary unitary ring (not necessarily commutative).  
	We denote by $e_{ij}\in M_d(R)$ the matrix which has $1$ at the $(i,j)$-th entry, and $0$ everywhere else. We have the 
	relations 
	$$
	e_{ij}e_{ks}=
	\begin{cases}
	e_{is} & \text{if $j=k$},\\ 
	0 & \text{otherwise}. 
	\end{cases}
	$$ 

\begin{lem}
	Let $R$ be a unitary ring for which every left ideal is principal. Let $\phi: M_d(A)\to M_d(R)$ 
	be an injective homomorphism. Then, up to conjugation by an element of $\GL_d(R)$, 
	we have $\phi(e_{ij})=e_{ij}$ for $1\leq i,j\leq d$. 
\end{lem}
\begin{proof}
	Let $M$ be a free left $R$-module of rank $d$. Let $e_1, \dots, e_d \in \End_R(M)$ be 
	non-zero elements which satisfy the following conditions 
	$$
	e_i\cdot e_i=e_i, \quad e_i\cdot e_j=0 \text{ if }i\neq j, \quad e_1+\cdots e_d=\id. 
	$$
	Let $M_i=e_i(M)\subset M$ be the image of $e_i$. Then $M_i\subset M$ is a non-zero $R$-submodule. 
	We observe the following:
	\begin{itemize}
		\item $e_i$ acts as $\id$ on $M_i$. (If $m\in M_i$ then $m=e_i(m')$ so $e_i m =e_i^2(m')=e_i(m')=m$.) 
		\item $e_i$ acts as $0$ on $M_j$. (Same argument as above.) 
		\item $M_i\cap M_j=0$. (If $m\in M_i\cap M_j$ then $e_i m=m$ since $m\in M_i$; on the other hand $e_i m=0$, since $m\in M_j$.) 
		\item $M=M_1+\cdots + M_d$. (Since $m=\id m= \sum e_im$.)  
	\end{itemize}
	Hence $M$ is an internal direct sum of the submodules $M_i$. We see that $M_i$ is a projective left $R$-module, and 
	since every left $R$-ideal is principal, $M_i$ is free. Since $M_i\neq 0$, $\rank_R M_i\geq 1$. Comparing the ranks of $\sum M_i$ 
	and $M$, we see that $\rank_R M_i=1$. If we choose the generators of $M_1, \dots, M_d$ as an $R$-basis of $M$, 
	then we get an isomorphism $\End_R(M)\cong M_d(R)$ such that $e_i=e_{ii}$. 
	
	Now let $e_{ij}'=\phi(e_{ij})$. Since $e_{ii}'$ satisfy the conditions listed above, after a conjugation corresponding to 
	mapping a given basis to the basis of the previous paragraph, we can assume $e_{ii}'=e_{ii}$. 
	Next, $e_{ij}'e_{jj}'=e_{ij}'$ and $e_{ii}' e_{ij}'=e_{ij}'$ shows that $e_{ij}'$ has zero entries except possibly at $(i,j)$-th entry, 
	which we denote $a_{ij}$. Since $e_{ij}' e_{ji}'=e_{ii}'$, we see that $a_{ij}a_{ji}=1$. Hence all $a_{ij}\in R^\times$.   
	After conjugating $\phi(M_d(A))$ by $\diag(a_{11}, a_{12}, \dots, a_{1d})$, we get $a_{1i}=a_{i1}=1$ for all $i$. 
	On the other hand, $a_{ij}=a_{i1}a_{1j}$, so $e_{ij}'$ become $e_{ij}$. 
\end{proof}

\begin{thm}\label{thmME-DSM}
	The category of of Drinfeld-Stuhler $M_d(A)$-modules over $L$ is equivalent to the category  
	of Drinfeld $A$-modules of rank $d$ over $L$ . 
\end{thm}
\begin{proof}
	Suppose $\phi$ is a Drinfeld-Stuhler $M_d(A)$-module. By the previous lemma, we can assume that $\phi(e_{ij})=e_{ij}$. 
	(Note that every left ideal of $L[\tau]$ is principal; cf. \cite[Cor. 1.6.3]{Goss}.) 
	The map 
	$\Phi:A\to L[\tau] $
	which sends $a$ to the non-zero entry of $\phi(a\cdot e_{11})$   
	is a Drinfeld $A$-module of rank $d$, as easily follows from considering the kernel of $\phi(\diag (a, 1, 1, \dots, 1))$. 
	Next, $a e_{ij}=e_{i1} (ae_{11}) e_{1j}$, which implies that $\phi(ae_{ij})$ is the matrix $\Phi_a e_{ij}$. Hence 
	$\phi$ arises from a unique Drinfeld $A$-module $\Phi$ of rank $d$ by the construction of Example \ref{exampleMdA}. 
	
	Now suppose $u:\Phi\to \Phi'$ is a morphism of Drinfeld modules, i.e., $u\in L[\tau]$ is such that $u\Phi_a=\Phi_a'u$ 
	for all $a\in A$. Mapping $u$ to $U:=\diag(u, \dots, u)$, we obtain a morphism $U\phi_b=\phi'_bU$, $b\in M_d(A)$, 
	of the corresponding Drinfeld-Stuhler modules. By an argument similar to the argument of the previous paragraph it is 
	not hard to check that any morphism $\phi\to \phi'$ arises in this manner. This proves the theorem. 
\end{proof}

\begin{rem}
	Given a unitary ring $R$ and a left $R$-module $M$, the direct sum $M^{\oplus d}$ is a left $M_d(R)$-module with  
	$M_d(R)$ acting on elements of $M^{\oplus d}$ as column vectors with entries in $M$. The functor $M\mapsto M^{\oplus d}$ 
	from the category of left $R$-modules to the category of left $M_d(R)$-modules is an equivalence of categories, known 
	as \textit{Morita equivalence}; cf. \cite{Reiner}.  The inverse functor is $M'\mapsto e_{11}M'$. 
\end{rem}

The Morita equivalence can be modified so that certain problems concerning Drinfeld-Stuhler $O_D$-modules 
reduce to the case of Drinfeld modules, even when $D\not\cong M_d(F)$.  This idea is due to Taelman \cite{TaelmanPhD}, 
who used it in the context of $O_D$-motives to prove a fact equivalent to the existence of analytic uniformization 
of Drinfeld-Stuhler modules. To end this section, we sketch Taelman's construction in the setting of Drinfeld-Stuhler modules, 
and indicate one application.

First, recall the following fact. Let $F'/F$ be a finite extension. 
Then $D\otimes_F F'$ is a central simple algebra over $F'$ and for a place $w$ of $F'$ 
over a place $v$ of $F$ we have (cf. \cite[Lem. A.3.2]{LaumonCDV})
$$
\inv_w(D\otimes_F F')=[F'_w:F_v]\cdot \inv_v(D) \in \Q/\Z. 
$$
Now suppose $F'=\F_{q^n}F$ is obtained by extending the constants. 
In this case $F'_w/F_v$ is unramified of degree $n/\gcd(n, \deg(v))$. 
Hence, using the above formula for the invariants of $D\otimes_F F'$, we see that there is $n$, e.g., 
$n=d\prod_{\inv_v(D)\not\in \Z} \deg(v)$, such that the invariants of $D\otimes_F F'$ at all places of $F'$ are $0$, 
which is equivalent to $D\otimes_F F' \cong M_d(F')$. This fact is known as \textit{Tsen's theorem}.  

Now let $\phi$ be a Drinfeld-Stuhler $O_D$-module over $L$,  and let $n$ be such that $F'=\F_{q^n}F$ splits $D$. Let $A'$ be the integral closure of $A$ 
in $F'$. Assume $\F_{q^n}\subset L$. 
Denote $\sigma=\tau^n$, and consider the composition 
$$
\phi':O_D\overset{\phi}{\To} M_d(L[\tau])\xrightarrow{\tau\mapsto \sigma} M_d(L[\sigma]).   
$$
(The second map is a formal substitution $\tau\mapsto \sigma$; it is not a homomorphism.) 
Note that $\phi'$ 
is not a Drinfeld-Stuhler module according to our definition, but the definition can be easily generalized so that $\phi'$ is 
a Drinfeld-Stuhler module of ``rank $n$''. Denote $O_{D'}:=O_D\otimes_{\F_q}\F_{q^n}$, and extend $\phi'$ to an embedding 
$$
\widetilde{\phi}':O_{D'} \To M_d(L[\sigma])
$$
by mapping $1\otimes\alpha\mapsto \diag(\alpha, \dots, \alpha)$. It is easy to check that $O_{D'}$ 
is a maximal $A'$-order in $D':=D\otimes_F F'\cong M_d(F')$, for example, by calculating its discriminant. Finally, using the Morita 
equivalence, one associates to $\widetilde{\phi}'$ a Drinfeld $A'$-module $\Phi'$ over $L$. 
(One technical complication that should be pointed out is that $O_{D'}$ might not be conjugate to $M_d(A')$ in $M_d(F')$ if $A'$ is not a P.I.D.,  
but for the Morita equivalence to work one only needs an idempotent $e$ in $O_{D'}$ which commutes with $\sigma$; cf. \cite[p. 68]{TaelmanPhD}.)

As for the promised application of the above construction, we prove Lemma \ref{lem1.7}. 
\begin{proof}[Proof of Lemma \ref{lem1.7}] 
Let $\phi$ and $\psi$ be Drinfeld-Stuhler $O_D$-modules over $L$ and 
$u:\phi\to \psi$ be a non-zero morphism. We want to show that $u$ is an isogeny. 
Without loss of generality, we assume that $L$ is algebraically closed.  
Explicitly, $u$ is a matrix in $M_d(L[\tau])$. Substituting $\sigma$ for $\tau$ 
in the entries of $u$, we get a matrix $u(\sigma)\in M_d(L[\sigma])$. It is clear that $u(\sigma)$ gives a morphism 
$\widetilde{\phi}'\to \widetilde{\psi}'$, and hence also a non-zero morphism $w: \Phi'\to \Psi'$. But now $w$ is 
a non-zero polynomial in $L[\sigma]$, so $\ker(w)$ is obviously finite, i.e., $w$ is an isogeny. Finally, it is easy to check 
that this implies that $u(\sigma)$, and thus also $u$ itself, is an isogeny. 
\end{proof}


\section{$O_D$-motives, $\sD$-elliptic sheaves and $O_D$-lattices }\label{sMotSht}

In this section, we introduce three categories closely related with the category of Drinfeld-Stuhler modules. These alternative 
points of view on Drinfeld-Stuhler modules will be important for the proofs of the  main results 
of this paper. None of the results of this section are original -- they are due to Anderson \cite{Anderson}, 
Laumon, Rapoport, Stuhler \cite{LRS}, and Taelman \cite{TaelmanPhD}. 
We keep the notation and assumptions of Section \ref{sBasics}. 
In particular,  
$L$ is an $A$-field. 
Let $O_D^\opp$ denote the opposite ring of $O_D$ (see \cite[p. 91]{Reiner}), i.e., $O_D^\opp$ is $O_D$ with the same addition 
but multiplication defined by $\alpha\ast \beta=\beta\cdot \alpha$, where $\beta\cdot \alpha$ is the multiplication in $O_D$. 

The first category is a variant of Anderson's motives. 

\begin{defn}
An \textit{$O_D$-motive} over $L$ is a left $O_D^\opp\otimes_{\F_q} L[\tau]$-module $M$ with the following properties 
(cf. \cite[p. 68]{TaelmanPhD}, \cite[p. 228]{LRS}):
\begin{itemize}
\item[(i)] $M$ is a locally free $O_D^\opp\otimes_{\F_q} L$-module of rank $1$.
\item[(ii)] $M$ is a free $L[\tau]$-module of rank $d$.
\item[(iii)] For all $a\in A$, 
$$
(a\otimes 1 -1\otimes \gamma(a)) \overline{M}\subset \tau \overline{M}, 
$$
where $\overline{M}:=M\otimes_LL^\alg$ is considered as a left $A\otimes_{\F_q} L^\alg[\tau]$-module. 
\end{itemize}
The morphisms between $O_D$-motives are the homomorphisms of $O_D^\opp\otimes_{\F_q} L[\tau]$-modules. 
We denote the corresponding category by $\DMot$. (An $O_D$-motive is a pure abelian Anderson $A$-motive, in the sense of \cite{vdHeiden} or \cite{BH}, 
of rank $d^2$, dimension $d$, and weight $1/d$; see \cite[$\S$9.2]{TaelmanPhD}.)
\end{defn}

Given a Drinfeld-Stuhler $O_D$-module $\phi$ over $L$, let 
$M(\phi)$ be the group $$\Hom_{\F_q}(\gm_{a, L}^d, \gm_{a, L})\cong L[\tau]^d$$ equipped with the unique $O_D^\opp\otimes_{\F_q} L[\tau]$-module 
structure such that 
$$
(\ell m)(e)=\ell(m(e)), \quad (\tau m)(e)=m(e)^q, \quad (b m)(e) = m(\phi(b)e),
$$
for all $e\in \gm_{a, L}^d$, $\ell\in L$, $b\in O_D$, and morphisms $m: \gm_{a, L}^d\to \gm_{a, L}$. 
It is easy to see that $M(\phi)$ is an $O_D$-motive. 

\begin{thm}\label{thmModMot}
The functor $\phi\mapsto M(\phi)$ gives an anti-equivalence of categories between the category of Drinfeld-Stuhler $O_D$-modules 
and $\DMot$. 
\end{thm} 
\begin{proof}
This can be proven by a slight modification of Anderson's method; see \cite[Thm. 2.3]{vdHeiden} or \cite[Thm. 3.5]{HartlIsog}. 
\end{proof}

The second category arises from $\sD$-elliptic sheaves mentioned in the introduction.  

\begin{defn}\label{defnDES}
Fix a maximal $\cO_C$-order $\sD$ in $D$ such that $H^0(C-\infty, \sD)=O_D$. 
A \textit{$\sD$-elliptic sheaf over $L$} is a sequence $\E=(\cE_i, j_i, t_i)_{i\in \Z}$, where 
$\cE_i$ is a locally-free $\cO_{C\otimes_{\F_q} L}$-module of rank $d^2$ equipped with a right action of $\sD$ 
which extends the $\cO_C$-action, and 
\begin{align*}
j_i &: \cE_i\hookrightarrow \cE_{i+1} \\ 
t_i &: \twist{\cE_i}:=(\mathrm{Id}_C\otimes \Frob_q)^\ast\cE_i \hookrightarrow \cE_{i+1}
\end{align*} 
are injective $\sD$-linear homomorphisms. Moreover, for each $i\in \Z$ the following conditions hold: 
\begin{enumerate}
\item[(i)] The diagram
$$
\xymatrix{\cE_i \ar[r]^{j_i} & \cE_{i+1}\\ \twist{\cE_{i-1}}
\ar[r]^{\twist{j_{i-1}}}\ar[u]^-{t_{i-1}} &
\twist{\cE_i}\ar[u]_-{t_i}}
$$
commutes;
\item[(ii)] $\cE_{i+d\cdot\deg(\infty)}=\cE_i\otimes_{\cO_C}\cO_C(\infty)$, and the inclusion
$$
\cE_i\xrightarrow{j_i}\cE_{i+1}\xrightarrow{j_{i+1}}\cdots \to
\cE_{i+d\cdot\deg(\infty)}=\cE_i\otimes_{\cO_C}\cO_C(\infty)
$$
is induced by $\cO_C\hookrightarrow \cO_C(\infty)$;
\item[(iii)] $\dim_L H^0(C\otimes L,\coker j_i)=d$;
\item[(iv)] $\cE_i/t_{i-1}(\twist{\cE_{i-1}})=z_\ast\cV_i$, where $\cV_i$
is a $d$-dimensional $L$-vector space, and $z$ is the morphism induced by $\gamma$:
$$
z: \Spec(L)\to \Spec(A)\to C. 
$$
\end{enumerate}
A \textit{morphism} between two 
$\sD$-elliptic sheaves over $L$ 
$$
\psi=(\psi_i)_{i\in \Z}:\E=(\cE_i, j_i, t_i)_{i\in \Z}\to
\E'=(\cE_i', j_i', t_i')_{i\in \Z}
$$
is a sequence of sheaf morphisms $\psi_i:\cE_i\to \cE'_{i+n}$ for
some fixed $n\in \Z$ which are compatible with the action of $\sD$
and commute with the morphisms $j_i$ and $t_i$:
$$
\psi_{i+1}\circ j_i=j_{i+n}'\circ \psi_i \quad \text{and}\quad
\psi_i\circ t_{i-1}=t_{i+n-1}'\circ \twist{\psi_{i-1}}.
$$
\end{defn}

Note that the group $\Z$ acts freely on the objects of the category of $\sD$-elliptic sheaves by ``shifting the indices'':
$$
n\cdot(\cE_i, j_i,t_i)_{i\in \Z}=(\cE_i', j_i',t_i')_{i\in \Z}
$$
with $\cE_i'=\cE_{i+n}$, $j_i'=j_{i+n}$, $t_i'=t_{i+n}$. 
Let $\DES/\Z$ be the quotient of the category of $\sD$-elliptic sheaves by this action of $\Z$. 

Let $\E=(\cE_i, j_i, t_i)_{i\in \Z}$ be a $\sD$-elliptic sheaf over $L$. Consider 
$$
M(\E):=H^0((C-\{\infty\})\otimes L, \cE_i). 
$$
This is independent of $i$ since $\mathrm{supp}(\cE_i/\cE_{i-1})\subset \{\infty\}\times \Spec(L)$. 
It is an $L[\tau]$-module, where the operation of $\tau$ is induced from $t_i: \twist{\cE}_i\to \cE_{i+1}$. 
In fact, $M(\E)$ is an $O_D$-motive; see \cite[(3.17)]{LRS}. 
\begin{thm}\label{thmShtMot}
The functor $\E\mapsto M(\E)$ gives an equivalence of $\DES/\Z$ with $\DMot$.  
\end{thm}
\begin{proof}
This is implicitly proven in \cite[(3.17)]{LRS} and explicitly in \cite[10.3.5]{TaelmanPhD}. We outline the main steps of the 
proof since part of this argument will be used later in the paper. 

First note that since $M(\E)$ does not depend on the choice 
of $\cE_i$, the map is indeed a functor from $\DES/\Z$ to $\DMot$. 
Next, let $W_\infty:=H^0(\Spec(O_\infty \hat{\otimes}L), \cE_0)$. From the definition 
of $\sD$-elliptic sheaf one deduces that $W_\infty$ has a natural structure of a free $L[\![\tau^{-1}]\!]$-module 
of rank $d$; see \cite[p. 231]{LRS}. 
In addition, $W_\infty$ is a right $\sD_\infty$-module so that we get an injective $\F_q$-algebra homomorphism 
$$ \varphi_\infty: \sD_\infty^\opp\to \End_{L[\![\tau^{-1}]\!]}(W_\infty), $$
and if we denote by 
$\pi_\infty$ a uniformizer of $O_\infty$ and $\tau_\infty=\tau^{\deg(\infty)}$, then 
$W_\infty$ has the property that $\tau_\infty^{-d}W_\infty=\pi_\infty W_\infty$. 

The pair $(M(\E), W_\infty)$ is a vector bundle of rank $d$ over the non-commutative projective line over $L$ 
in the sense of \cite[(3.13)]{LRS}. Hence, by \cite[(3.16)]{LRS}, 
$$
(M(\E), W_\infty)\cong \cO(n_1)\oplus\cdots \oplus \cO(n_d),
$$
where $\cO(n)=(L[\tau], \tau^n L[\![\tau^{-1}]\!])$. Since $(M(\E), W_\infty)$ is equipped with a 
coherent right $\sD$-action (cf. \cite[(3.14)]{LRS}), we have $n_1=\cdots=n_d$. Hence 
$(M(\E), W_\infty)\cong \cO(n)^{\oplus d}$ for some $n\in \Z$. If we define $W_\infty'=H^0(\Spec(O_\infty \hat{\otimes}L), \cE_i)$, 
then $(M(\E), W_\infty')$ is again a vector bundle of rank $d$ over the non-commutative projective line. 
Moreover $(M(\E), W_\infty')=(M(\E), \tau^iW_\infty)$; see \cite[p. 235]{LRS}. 
Hence, up to the action of $\Z$, $M(\E)$ uniquely determines the vector bundle $(M(\E), W_\infty)$. 
On the other hand, by \cite[(3.17)]{LRS}, the vector bundle $(M(\E), W_\infty)$ with its coherent $\sD$-action uniquely 
determines $\E$ and any $O_D$-motive is isomorphic to $M(\E)$ for some $\E$. This proves that the functor 
in question is fully faithful and essentially surjective. 
\end{proof}

The third category arises in the theory of analytic uniformization of Drinfeld-Stuhler modules. Let $\C_\infty$ be the completion 
of an algebraic closure of $\Fi$. Let $\phi$ be a Drinfeld-Stuhler $O_D$-module over $\C_\infty$. 
By fixing an isomorphism $\Lie(\gm_{a, \C_\infty}^d)\cong \C_\infty^d$, we get an action of $O_D$ on $\C_\infty^d$ via $\partial_\phi$. 
\begin{thm}\label{thmUnifDSM}
There is a discrete $O_D$-submodule $\La_\phi$ of $\C_\infty^d$,   
which is locally free of rank $1$,  and 
an entire $\F_q$-linear function $\exp_\phi: \C_\infty^d\to \C_\infty^d$, which is surjective 
with kernel $\La_\phi$, such that for any $b\in O_D$ the following diagram is commutative:
$$
\xymatrix{
0 \ar[r] & \La_\phi \ar[r] \ar[d]_-{\partial_{\phi}(b)}&  \C_\infty^d \ar[r]^-{\exp_\phi} \ar[d]_-{\partial_{\phi}(b)} & 
\C_\infty^d \ar[r] \ar[d]_-{\phi_b}& 0  \\ 
0 \ar[r] & \La_\phi \ar[r] &  \C_\infty^d \ar[r]^-{\exp_\phi} & \C_\infty^d \ar[r] & 0. 
}
$$
\end{thm}
\begin{proof}
The exponential function $\exp_\phi$ is the function constructed by Anderson in \cite[$\S$2]{Anderson}. 
The existence of $\La_\phi$ (which is equivalent to the surjectivity of $\exp_\phi$ by \cite[Thm. 4]{Anderson}) 
was proved by Taelman \cite[$\S\S$9-10]{TaelmanPhD} in the terminology of $O_D$-motives. 
A crucial point in Taelman's proof 
is the use of Morita equivalence (see $\S$\ref{sME}), which reduces the proof 
to the analytic uniformization of Drinfeld modules (already known by the work of Drinfeld \cite{Drinfeld} 
and Anderson \cite{Anderson}).  
\end{proof}

\begin{cor}\label{cor3.6}
The ring $\End(\phi)$ is canonically isomorphic to the ring $$\End(\La_\phi):=\{c\in \C_\infty\mid c\La_\phi\subseteq \La_\phi\}.$$ 
\end{cor}
\begin{proof} The functorial properties of $\exp_\phi$ (cf. \cite[p. 473]{Anderson}) imply that $\partial$ maps $\End(\phi)$ isomorphically to the 
ring  
$$
\{P\in M_d(\C_\infty)\mid P\La\subseteq \La, P\partial_\phi(b)=\partial_\phi(b)P\text{ for all }b\in O_D\}. 
$$
Since any matrix which commutes with $\partial_\phi(O_D)$ must be a scalar, we get the desired isomorphism. 
\end{proof}

Now suppose $\C_\infty^d$ is equipped with an action of $O_D$ via some embedding $\iota:O_D\to M_d(\C_\infty)$. 
Suppose there is a discrete $\iota(O_D)$-submodule $\La\subset \C_\infty^d$ which is locally free of rank one. Then 
there is a unique Drinfeld-Stuhler $O_D$-module such that $\iota=\partial_\phi$ and $\La=\La_\phi$; 
this follows from \cite[$\S$10.1.3]{TaelmanPhD}, which itself crucially relies on \cite[Thm. 6]{Anderson}.  
Hence the category of Drinfeld-Shuhler modules over $\C_\infty$ is equivalent to the category of $O_D$-lattices as above.  
One can use this equivalence to give an analytic description of the set of isomorphism classes of Drinfeld-Shuhler modules 
over $\C_\infty$ as follows: Let
$$\Omega^d=\p^{d-1}(\C_\infty)-\bigcup_H H(\C_\infty)$$  be the Drinfeld 
symmetric space, where $H$ 
runs through the set of $\Fi$-rational hyperplanes in $\p^{d-1}(\C_\infty)$. Similar to the ring of finite ad\`eles 
$$
\A_f= \{(a_v)\in {\prod_{v\neq \infty}} F_v \mid a_v\in A_v \text{ for almost all $v$}\}, 
$$
define 
$$
D(\A_f) 
= \{(a_v)\in {\prod_{v\neq \infty}} D_v \mid a_v\in O_D\otimes_A A_v \text{ for almost all $v$}\}. 
$$
Let $\hat{A}:=\prod_{v\neq \infty}A_v$ and 
$\widehat{O}_D:=\prod_{v\neq \infty} O_D\otimes_A A_v$. 
We embed $D$ in $D(\A_f)$ diagonally. 
Fixing an isomorphism $D_\infty\cong M_d(\Fi)$, identifies $D^\times$ with a subgroup of $\GL_d(\Fi)$ and 
therefore induces an action of $D^\times$ on $\Omega$. 

\begin{prop}\label{prop1.12}
There is a one-to-one correspondence between 
the set of isomorphism classes of Drinfeld-Shuhler $O_D$-modules 
over $\C_\infty$ and the double coset space 
$$
D^\times\bs \Omega^d\times D(\A_f)^\times /\widehat{O}_D^\times,
$$
where $D^\times$ acts on both $\Omega^d$ and $D(\A_f)^\times$ on the left, and $\widehat{O}_D^\times$ 
acts on $D(\A_f)^\times$ on the right: 
$$
\gamma \cdot (z, \alpha)\cdot k= (\gamma z, \gamma \alpha k), \quad \gamma\in D^\times, \quad 
z\in \Omega^d, \quad \alpha\in D(\A_f)^\times, \quad k\in \widehat{O}_D^\times. 
$$
\end{prop}
\begin{proof} This can be proved by a standard argument \cite[p. 74]{TaelmanPhD} (see also \cite[Thm. 4.4.11]{BS}). 
We recall this argument,  since we will use it later on. 
Let $\La\subset \C_\infty^d$ be an $O_D$-lattice, where $D$ acts on $\C_\infty^d$ via the fixed isomorphism $D_\infty\cong M_d(\Fi)$. 
The $F$-span $F\La$ is a free module over $D$ of rank $1$. 
A choice of generator of this module defines a point in $\p^{d-1}(\C_\infty)$. One checks that this point lies in $\Omega^d$ 
if and only if $\La$ is discrete. The 
embedding $\La\subset F\La=D$ can be tensored to an embedding $\hat{A}\La\subset D(\A_f)$ 
and the former can be recovered from the latter as $\La=\hat{A}\La\cap D$. Now $\hat{A}\La$ is a 
locally free module over $\widehat{O}_D$. Since all such modules are free, we conclude that 
the locally free $O_D$-submodules $\La\subset D$ 
of rank one are in bijection with the free rank one $\widehat{O}_D$-submodules of $D(\A_f)$ 
and the latter are in bijection with $D(\A_f)^\times/\widehat{O}_D^\times$. Finally, moding out 
by the choice of the generator of $F\La$, that is, by $D^\times$, we get the desired one-to-one correspondence.  
\end{proof}


\section{Complex multiplication}\label{sCM} 

This section contains our main results about the endomorphism rings of Drinfeld-Stuhler modules. The proofs rely on the
concepts introduced in Section \ref{sMotSht}. 

\begin{thm}\label{thmEnd} Let $\phi$ be a Drinfeld-Stuhler $O_D$-module over an $A$-field $L$. Then: 
\begin{enumerate}
\item $\End_L(\phi)$ is a projective $A$-module of rank $\leq d^2$. 
\item $\End_L(\phi)\otimes_A \Fi$ is isomorphic to a subalgebra of the central division algebra over $\Fi$ with invariant $-1/d$.
\item If $L$ has generic $A$-characteristic, then $\End_L(\phi)$ is an $A$-order 
in an imaginary field extension of $F$ which embeds into $D$. In particular, $\End_L(\phi)$ 
is commutative and its rank over $A$ divides $d$. 
\item The automorphism group $\Aut_L(\phi):=\End_L(\phi)^\times$ is isomorphic to $\F_{q^s}^\times$ for some $s$ dividing $d$. 
\end{enumerate} 
\end{thm}
\begin{proof}
It is enough to prove (1), (2) and (3) after extending $L$ to its algebraic closure, so we will assume that 
$L$ is algebraically closed. 

Since the $O_D$-motive $M(\phi)$ associated to $\phi$ is an Anderson $A$-motive
of dimension $d$ and rank $d^2$, the argument in \cite[$\S$1.7]{Anderson} implies that $\End_{A\otimes L[\tau]}(M(\phi))$  
is a projective $A$-module of rank $\leq d^4$ (see also \cite[Thm. 9.5]{BH} and \cite[Cor. 2.6]{HartlIsog}). 
Hence, thanks to Theorem \ref{thmModMot}, $\End(\phi)$ is a projective $A$-module of rank $\leq d^4$.  

Let $W_\infty$ be the $\sD_\infty\otimes L[\![\tau^{-1}]\!]$-module attached to $\phi$ in the proof of Theorem \ref{thmShtMot}. 
As we discussed, $W_\infty$ is well-defined up to the shifts $W_\infty\mapsto \tau W_\infty$. 
Since $\sD_\infty\cong \M_d(O_\infty)$, using the Morita equivalence, cf. \cite[p. 262]{LRS}, 
one concludes that $W_\infty$ is equivalent to an $O_\infty\otimes L[\![\tau^{-1}]\!]$-module $W_\infty'$
which is free of rank $1$ over $O_\infty$, free of rank $1$ over $L[\![\tau^{-1}]\!]$, and   
$\tau_\infty^{-d}W_\infty'=\pi_\infty W_\infty'$.
From $W_\infty'$ we get an $\F_q$-algebra homomorphism 
$$
\phi_\infty: O_\infty\to \End_{L[\![\tau^{-1}]\!]}\left(W_\infty'\right)=L[\![\tau^{-1}]\!], 
\quad \phi_\infty(\pi_\infty)=\tau_\infty^{-d}.  
$$
Thus, 
$$
\End_{O_\infty\otimes L[\![\tau^{-1}]\!]}(W_\infty')^\opp = 
\End(\phi_\infty)=\{ f\in L[\![\tau^{-1}]\!]\ |\ f\phi_\infty(b)=\phi_\infty(b)f \text{ for all }b\in O_\infty\}. 
$$
Since $O_\infty=\F_{q^{\deg(\infty)}}[\![\pi_\infty]\!]$, the image of $O_\infty$ under 
$\phi_\infty$ is the subring $\F_{q^{\deg(\infty)}}[\![\tau_\infty^{-d}]\!]$ of $L[\![\tau^{-1}]\!]$. Now it is easy to see that 
$$
\End(\phi_\infty)\cong \F_{q^{d\deg(\infty)}}[\![\tau_\infty^{-1}]\!],
$$
which is the maximal order in the central division algebra over $\Fi$ with invariant $-1/d$; cf. \cite[Appendix B]{LRS}. 
Definition 3.14 and Theorem 3.17 in \cite{LRS} imply that $\End(\phi)$ acts faithfully on $W_\infty'$, 
and this action gives an embedding $\End(\phi)\otimes_A \Fi \hookrightarrow \End(\phi_\infty)\otimes_{O_\infty}\Fi$ 
(see also \cite[Thm. 8.6]{BH}). Since 
$\rank_{O_\infty}\End(\phi_\infty) = d^2$, we get $\rank_{A}\End(\phi)\leq d^2$. This proves (1) and (2). 

To prove (3), note that $\phi$ is defined over some finitely generated subfield of $L$ which can be embedded 
into $\C_\infty$. So, without loss of generality, we assume $L= \C_\infty$. Combining (1) and (2) with Lemma \ref{lemECom} 
already implies that $\End(\phi)$ is an $A$-order in an imaginary field extension of $F$. We need to show 
that $\End(\phi)$ embeds into $D$. Let $\La_\phi$ be the $O_D$-lattice associated to $\phi$ 
by Theorem \ref{thmUnifDSM}. By Corollary \ref{cor3.6}, $\alpha\in \End(\phi)$ corresponds 
to $c\in \C_\infty$ such that $c\La_\phi\subseteq \La_\phi$. On the other hand, the $F$-span $F\La_\phi$ is a free 
module over $D$ of rank $1$, so $c$ corresponds to a unique element of $D$.  
Mapping $\alpha$ to that element, gives an embedding $\End(\phi)\hookrightarrow D$. 
Finally, the rank of $\End(\phi)$ over $A$ is equal to the degree of $\End(\phi)\otimes_F F$ over $F$, and  
it is well-known that a subfield of $D$ containing $F$ has degree over $F$ dividing $d$; cf. \cite[$\S$7]{Reiner}. 

To prove (4), note that we have established that $E:=\End_L(\phi)$ is an $A$-order in the 
division algebra $H:=E\otimes_A F$ over $F$. 
In this situation, $\alpha\in E$ is a unit if and only if $\Nr_{H/F}(\alpha)\in A^\times\cong \F_q^\times$, where 
$\Nr_{H/F}$ is the norm on $H$; cf. \cite[p. 224]{Reiner}.  We also proved that 
$H_\infty:=H\otimes \Fi$ is a subalgebra of the 
central division algebra over $\Fi$ with invariant $-1/d$. 
It is known that 
$\cH=\{\alpha\in H_\infty\ |\ \Nr_{H_\infty/\Fi}(\alpha)\in O_\infty\}$ 
is the unique maximal order of $H_\infty$, $\cM=\{\alpha\in H_\infty\ |\ \Nr_{H_\infty/\Fi}(\alpha)\in \fp_\infty\}$ 
is the unique maximal two-sided ideal of $\cH$, and $\cH/\cM$ is a subfield of $\F_{q^d}$; see (12.8), (13.2), (14.3) in \cite{Reiner}. 
Since the norm of any element of $E$ is in $A$, the subring $k:=E\cap \cH$ maps injectively into $\cH/\cM\hookrightarrow \F_{q^d}$.  
Thus, $k$ is isomorphic to a subfield of $\F_{q^d}$. Finally, 
it is clear that $\Aut_L(\phi)=k^\times$. 
\end{proof}

\begin{rem}\label{rem4.2}
	When $D=M_d(F)$, via the equivalence of $\S$\ref{sME}, 
	the statements of Theorem \ref{thmEnd} are equivalent to some well-known facts about the endomorphism 
	rings of Drinfeld $A$-modules of rank $d$; cf. \cite{GreenBook}.   
\end{rem}

\begin{example}\label{exampleCM} Let $\phi$ be a Drinfeld-Stuhler $O_D$-module over  
an algebraically closed field $L$ of generic $A$-characteristic. From Theorem \ref{thmEnd} we know that $\End(\phi)$ 
is an $A$-order in an imaginary field extension of $F$ of degree dividing $d$. 
We show that this bound is the best possible. Consider $\phi$ from Example \ref{example9}. Let $\Phi$ 
be the rank-1 Drinfeld $O_K$-module over $L$ from the same example. 
Let 
$$
E:=\{\diag(\Phi_\alpha, \dots, \Phi_\alpha)\ |\ \alpha\in O_K\}\subset M_d(L[\tau]). 
$$
It is clear that $E\cong O_K$. One easily checks that the elements of $E$ commute with $\phi_\alpha$, $\alpha\in O_K$, 
and $\phi_z$. Therefore $E\subseteq \End(\phi)$. Since $O_K$ is a maximal $A$-order in $K$, Theorem \ref{thmEnd} 
implies that $\End(\phi)\cong O_K$. 
\end{example}

\begin{defn}\label{defnCM}
For the rest of this section, unless indicated otherwise, $K$ denotes an imaginary field extension of $F$ of degree $d$, and $O_K$ 
denotes the integral closure of $A$ in $K$. An \textit{$A$-order} $E$ in $K$ is an $A$-subalgebra of $O_K$, 
which has the same unity element as $O_K$ and such that $O_K/E$ has finite cardinality. 
We say that a Drinfeld-Stuhler $O_D$-module $\phi$ over a field of generic $A$-characteristic 
has \textit{complex multiplication} by $E$ (or has \textit{CM}, for short) 
if $E\cong\End(\phi)$. Note that in that case $K$ necessarily embeds into $D$. A \textit{CM subfield} of $D$ 
is a commutative subfield of $D$ which is an imaginary extension of $F$ of degree $d$.  
\end{defn}

\begin{lem} 
Let $\phi$ be a Drinfeld-Stuhler $O_D$-module over a field $L$ of generic $A$-characteristic. Assume $E\subset \End_L(\phi)$. 
There is a Drinfeld-Stuhler $O_D$-module $\psi$ which is isogenous to $\phi$ over $L$ and $\End_L(\psi)\cong O_K$.  
\end{lem}
\begin{proof}
Let $\fc$ be the conductor of $E$, i.e., the largest ideal of $O_K$ which is also an ideal of $E$. Let $H:=\bigcap_{c\in \fc} \ker(c)$, 
where the intersection is taken in $\gm_{a, L}^d$. Since the action of $O_D$ on $\gm_{a, L}^d$ commutes with $E$, 
the finite \'etale subgroup scheme $H$ of $\gm_{a, L}^d$ is invariant under $\phi(O_D)$. 
From the discussion on page 155 in \cite{Goss} it follows that there is $u\in \End_{\F_q}(\gm_{a, L}^d)$ 
with $\ker(u)=H$. 
Let $b\in O_D$. Consider the endomorphism $u\phi_b$ of $\gm_{a, L}^d$. 
Since $H$ is invariant under $\phi(O_D)$, we have $H\subseteq \ker(u\phi_b)$. 
Then we can factor $u\phi_b$ as $\psi_bu$ for some $\psi_b\in \End_{\F_q}(\gm_{a, L}^d)$. 
(See Proposition 5.2 and Corollary 5.3 in \cite{HartlIsog} for some relevant scheme-theoretic facts about 
morphisms $\gm_{a, L}^d\to \gm_{a, L}^d$; in particular, note that $\gm_{a, L}^d/H\cong \gm_{a, L}^d$.) 
It is easy to see that $b\mapsto \psi_b$ gives an embedding $O_D\to M_d(L[\tau])$ and $\# \psi[b]=\# \phi[b]$. 
Since $\partial(u)\in M_d(L)$ is an invertible matrix, and $\partial_{\phi}(a)$ ($a\in A$) is the scalar matrix 
$\diag(\gamma(a), \dots, \gamma(a))$, we also get that $\partial_{\psi}(a)=\diag(\gamma(a), \dots, \gamma(a))$. 
Thus, there is a Drinfeld-Stuhler $O_D$-module $\psi$ over $L$ and an isogeny $u:\phi\to \psi$ whose kernel is $H$. 
Now one can apply the argument in the proof of \cite[Prop. 4.7.19]{Goss} to deduce that $\End(\psi)\cong O_K$. 
\end{proof}

We further investigate the properties of Drinfeld-Stuhler modules with CM using analytic uniformization. 
We fix an embedding $D^\times\to \GL_d(\Fi)$ through which $D^\times$ acts on $\Omega$. 
For $(z, \alpha)\in \Omega^d\times D(\A_f)^\times $, let $\La_{(z, \alpha)}$ be the $O_D$-lattice corresponding to $(z, \alpha)$; 
see Proposition \ref{prop1.12}. Let 
$K_z^\times:=\{\gamma\in D^\times\mid \gamma z=z\}$. 
From now on we will implicitly assume that $D$ is a division algebra.

\begin{lem}\label{lem4.5} $K_z:=K_z^\times \cup \{0\}$ is a subfield of $D$ and 
$\End(\La_{(z, \alpha)})=K_z\cap \alpha \widehat{O}_D \alpha^{-1}$. 
\end{lem}
\begin{proof}
Let $\tilde{z}\in \C_\infty^d$ be an element mapping to $z$; such $\tilde{z}$ is well-defined 
up to a scalar multiple. 
Denote $O= D\cap \alpha \widehat{O}_D\alpha^{-1}$. 
The lattice $\La=O\tilde{z}$ 
is in the isomorphism class of $\La_{(z, \alpha)}$. 
We have 
$$
c\in \End(\La)\quad \Longleftrightarrow\quad c\La\subset \La\quad \Longleftrightarrow\quad c O\tilde{z}\subset O\tilde{z} 
\quad\Longleftrightarrow\quad
 O c\tilde{z}\subset O\tilde{z},
$$
where $c\in \C_\infty$ acts on $\tilde{z}$ as a scalar matrix. The inclusion $O c\tilde{z}\subset O\tilde{z}$ is equivalent to 
the existence of $\gamma\in O$ such that $\gamma \tilde{z}=c\tilde{z}$. This $\gamma$ obviously fixes $z$, and 
since $\gamma\in \alpha \widehat{O}_D\alpha^{-1}$, we get $\gamma\in K_z\cap \alpha \widehat{O}_D \alpha^{-1}=:E_{(z, \alpha)}$. 
Conversely, suppose $\gamma\in E_{(z, \alpha)}$, so $\gamma\in O$ and $\gamma \tilde{z}=c\tilde{z}$ 
for some non-zero $c\in \C_\infty$ (because $\gamma\in K_z$). Reversing the previous argument we see that $c\in \End(\La)$. 

Observe that $E_{(z, \alpha)}$ is a subring of $D$ since 
for $\gamma, \gamma'\in E_{(z, \alpha)}$ with $\gamma \tilde{z}=c \tilde{z}$,  
$\gamma' \tilde{z}=c' \tilde{z}$, we have $(\gamma+\gamma') \tilde{z}=(c+c') \tilde{z}$. 
Hence $K_z=E_{(z, \alpha)}\otimes_A F$ is a commutative subalgebra of $D$, i.e., $K_z$ is a subfield of $D$. 
Since the map $E_{(z, \alpha)}\to \End(\La)$, $\gamma\mapsto c$, 
is a homomorphism which extends to $K_z\to \C_\infty$, it must be injective. 
 But we have seen that $E_{(z, \alpha)}\to \End(\La)$ is also surjective, thus it is an isomorphism. 
\end{proof}

\begin{rem}
For any $\alpha\in D(\A_f)^\times$ and a CM field $K\subset D$, the intersection $K\cap \alpha  \widehat{O}_D \alpha^{-1}$ 
is an $A$-order in $K$. To prove this, first observe that $D\cap \alpha  \widehat{O}_D \alpha^{-1}$  
is a maximal order in $D$. Hence it is enough to prove that for any maximal order $\cM$ in $D$ the intersection 
$E:=K\cap \cM$ is an $A$-order. 
It is clear that $A\subset E$. By Exercise 4, p. 131, \cite{Reiner}, there is a maximal order $\cM'$ in $D$ 
such that $K\cap \cM'=O_K$. It is easy to see that $\cM'':=\cM\cap \cM'$ is an $A$-order in $D$ (in fact, 
it is a hereditary order; cf. \cite[$\S$40]{Reiner}). 
Hence $\cM''$ has finite index in $\cM'$. 
On the other hand, since $E=O_K\cap \cM''$, under the natural homomorphism $\cM'\to \cM'/\cM''$ the 
module $O_K$ maps onto $O_K/E$. Thus, $E$ has finite index in $O_K$, i.e., is an order. 
\end{rem}

\begin{lem}\label{lemFixPts}
Let $K$ be a CM subfield of $D$. The number of fixed points of $K^\times$ in $\Omega^d$ 
is non-zero and is at most $d$. 
\end{lem}
\begin{proof}
Since $F$ has transcendence degree $1$ over $\F_q$, we can find a primitive element $\gamma\in K$ 
such that $K=F(\gamma)$; cf. \cite{BMcL}. It is enough to prove that $\gamma$ has at least one and at most $d$ fixed points in $\Omega^d$. 
Note that the minimal polynomial of $\gamma$ over $F$ has degree $d$ and divides the characteristic polynomial 
of $\gamma$ considered as an element of $\GL_d(\Fi)$. 
Thus, the minimal polynomial is equal to the characteristic polynomial. 
The claim then follows from the fact that a matrix in $\GL_d(\C_\infty)$, whose characteristic and minimal 
polynomials are equal, has at least one and at most $d$ eigenvectors, up to scaling. 
\end{proof}

\begin{notn}
Let $K$ be a CM subfield of $D$, and $E$ be an $A$-order in $K$. Let 
$$
\T_E:=\{\alpha\in D(\A_f)^\times\mid K\cap \alpha  \widehat{O}_D \alpha^{-1}=E\}. 
$$
It is easy to check that $K^\times$ acts on $\T_E$ from the left by multiplication and $\widehat{O}_D^\times$ acts 
from the right. It is known that $\T_E$ is non-empty and 
the double coset space $K^\times \bs \T_E/\widehat{O}_D^\times$ has finite cardinality 
divisible by the class number of $E$; cf. \cite[pp. 92-93]{Vigneras}. 
The elements of $\T_E$ correspond to optimal embeddings of $K$ into the maximal orders of $D$ with respect to $E$.
\end{notn}

\begin{thm}\label{propCM} Let $S_K$ be the set of fixed points of $K^\times$ in $\Omega^d$. We have:
\begin{enumerate}
\item Up to isomorphism, the number of Drinfeld-Stuhler $O_D$-modules over $\C_\infty$ 
having CM by $E$ is equal to 
$\# (K^\times \bs S_K\times \T_E/\widehat{O}_D^\times)$. In particular, that number is 
finite and non-zero. 
\item A Drinfeld-Stuhler $O_D$-module having CM by $O_K$ can be defined over the Hilbert class field of $K$. 
\end{enumerate}
\end{thm}
\begin{proof} (1) In our set-up, we have fixed an embedding of $K$ into $D$. 
For each $(z, \alpha)\in S_K\times \T_E$ we 
have $\End(\La_{(z, \alpha)})=K_z\cap \alpha\widehat{O}_D\alpha^{-1}=K\cap \alpha\widehat{O}_D\alpha^{-1}=E$; cf. Lemma \ref{lem4.5}. 
Note that for any $\gamma\in D^\times$, we have $K_{\gamma z}=\gamma K_z\gamma^{-1}$, and 
so $$K_{\gamma z}\cap \gamma\alpha\widehat{O}_D\alpha^{-1}\gamma^{-1}=\gamma (K\cap \alpha\widehat{O}_D\alpha^{-1})\gamma^{-1}=
\gamma E\gamma^{-1}\cong E,$$
which implies $\End(\La_{\gamma(z, \alpha)})\cong \End(\La_{(z, \alpha)})$. 

Now suppose $(z, \alpha)\in \Omega^d\times D(\A_f)^\times$ is such that $\End(\La_{(z, \alpha)})\cong E$. Then $K_z$ 
must be isomorphic to $K$, so $K_z$ is another embedding of $K$ into $D$.  
By the Skolem-Noether theorem \cite[(7.21)]{Reiner}, two embeddings $K\rightrightarrows D$ 
differ by an inner automorphism of $D$. Thus, there is $\gamma\in D^\times$ such that $K_z=\gamma K\gamma^{-1}$ 
and $\End(\La_{(z, \alpha)})= \gamma E\gamma^{-1}$. 
This implies that we can find $z'\in S_K$ such that $\gamma z'=z$. We also have 
$\gamma K\gamma^{-1}\cap \alpha\widehat{O}_D\alpha^{-1} = \gamma E \gamma^{-1}$, which 
implies $\gamma^{-1}\alpha\in \T_E$. Hence we can find $\alpha'\in \T_E$ such that $\alpha=\gamma \alpha'$. 
Overall, we conclude that $(z, \alpha)=\gamma (z', \alpha')$ for some $(z', \alpha')\in S_K\times \T_E$.  
The stabilizer in $D^\times$ of any $z\in S_K$ is $K^\times$. Hence the set of images in 
$D^\times\bs \Omega^d\times D(\A_f)^\times /\widehat{O}_D^\times$ of $(z, \alpha)\in \Omega^d\times D(\A_f)^\times$ 
with CM by $E$ is the double coset space $K^\times \bs S_K\times \T_E/\widehat{O}_D^\times$. 

(2) Let $\phi$ be a Drinfeld-Stuhler $O_D$-module with $\End(\phi)\cong O_K$. 
Let $M(\phi)$ be the $O_D$-motive associated to $\phi$. 
By definition, the action of $O_K$ on $\gm_{a, \C_\infty}^d$ commutes with $\phi(O_D)$, hence $M(\phi)$ 
is an $O_D^\opp\otimes_A O_K$-module. 
On the other hand, $O_D^\opp\otimes_A O_K$ is an $A$-order in $D^\opp\otimes_F K\cong M_d(K)$; cf. Exercise 6 on page 131 of \cite{Reiner}. 
Computing the discriminants, one checks that $O_D^\opp\otimes_A O_K$ is a maximal order in $M_d(K)$. 
By the Morita equivalence (cf. \cite[p. 262]{LRS} and \cite[p. 68]{TaelmanPhD}), $M(\phi)$ 
is equivalent to an $O_K$-motive $M'$ of rank $1$ and dimension $1$ (as defined in \cite{vdHeiden}). 
Through a generalization of 
Anderson's result (cf. \cite[Thm. 2.9]{vdHeiden}), $M'$ corresponds to a Drinfeld $O_K$-module $\Phi$ of rank $1$. 
Since $\Phi$ can be defined over $H$ (see \cite[$\S$8]{HayesCFT}),  $\phi$ also can be defined over $H$. 
\end{proof}

\begin{rem} If $d=2$, $K$ is a separable quadratic extension of $F$, and $E=O_K$, then (cf. \cite[p. 94]{Vigneras})
 $$
 \#(K^\times \bs S_K\times \T_{O_K}/\widehat{O}_D^\times)= \#\Pic(O_K)\prod_{\fp\mid\fr(D)}\left(1-\left(\frac{K}{\fp}\right)\right), 
 $$
 where 
 $$
 \left(\frac{K}{\fp}\right)=
\begin{cases}
-1, & \text{if $\fp$ remains inert in $K$};\\
1, & \text{if $\fp$ splits in $K$};\\
0, & \text{if $\fp$ ramifies in $K$}. 
\end{cases} 
$$
 \end{rem}

From Theorem \ref{thmEnd} (4) we know that $\Aut(\phi)\cong \F_{q^s}^\times$ for some $s\mid d$. 
Moreover, Examples \ref{example12} and \ref{exampleCM} 
show that there do exist Drinfeld-Stuhler modules with maximal possible automorphism group $\F_{q^d}^\times$. 
We can now give necessary and sufficient conditions for the existence of Drinfeld-Stuhler modules with given automorphism 
group. 

\begin{cor}\label{corLargeAut} Let $L$ be an algebraically closed field of generic $A$-characteristic. 
Given $s\mid d$, there is a Drinfeld-Stuhler 
module $\phi$ over $L$ with $\Aut(\phi)\cong \F_{q^s}^\times$  if and only if 
\begin{itemize}
\item[(i)] $s$ is coprime to $\deg(\infty)$ and 
\item[(ii)] $\frac{d}{\gcd(s, \deg(v))}\inv_v(D)\equiv 0\ (\mod\ {\Z})$ for all $v\in \Ram(D)$.
\end{itemize}
\end{cor}
\begin{proof} Let $K:=\F_{q^s}F$. Note that $K/F$ is a field extension of $F$ of degree $s$. 
If $w$ is a place of $K$ over a place $v$ of $F$, then the local extension $K_w/F_v$ 
is unramified of degree $s/\gcd(s, \deg(v))$. This implies that $K$ is imaginary if and only if (i) holds.  
On the other hand, by \cite[Prop. A.3.4]{LaumonCDV}, $K$ embeds into $D$ if and only if $[K:F]$ divides $d$ and 
$$
\frac{d}{[K:F]}[K_w:F_v] \inv_v(D)\equiv 0\ (\mod\ {\Z}) \quad \text{for all $v\in \Ram(D)$}. 
$$ 
Thus, $K\hookrightarrow D$ if and only if (ii) holds. 

Suppose there is $\phi$ over $L$ with $\Aut(\phi)\cong \F_{q^s}^\times$. Then $K\subset \End(\phi)\otimes_A F$,  and
Theorem \ref{thmEnd} (3) implies that $K$ is isomorphic to an imaginary subfield of $D$. 
Conversely, suppose $K$ is an imaginary subfield of $D$. We can extend $K$ to a maximal 
subfield $K'$ of $D$ 
so that $\F_{q^s}$ is algebraically closed in $K'$ and the place of $K$ over $\infty$ does not split in $K'$. 
Then $K'$ is a CM subfield of $D$ and $O_{K'}^\times\cong \F_{q^s}^\times$. 
By a ``Lefschetz principle'' type argument we can assume 
$L=\C_\infty$. The existence of $\phi$ with $\Aut(\phi)\cong \F_{q^s}^\times$ then follows from 
Theorem \ref{propCM} (1) by taking $E=O_{K'}$. 
\end{proof}

Next, in a special case, we compute the number of isomorphism classes 
of Drinfeld-Stuhler modules over $\C_\infty$ with maximal possible automorphism group $\F_{q^d}^\times$. 
In principle, this amounts to an explicit computation of the order of the double coset space in Theorem \ref{propCM} (1), 
but we will take a somewhat different approach which also provides a group-theoretic interpretation of this number. 

\begin{prop}\label{propLargeAut} Assume $A=\F_q[T]$ and $d$ is coprime to $\deg(v)$ for all $v\in \Ram(D)$.
Up to isomorphisms, there are $d^{\# \Ram(D)}$ Drinfeld-Stuhler modules $\phi$ over $\C_\infty$ with 
 $\Aut(\phi)\cong \F_{q^d}^\times$. 
\end{prop}
\begin{proof} Note that in this case $\deg(\infty)=1$, so assumption (i) in Corollary \ref{corLargeAut} 
is automatically satisfied. Also note that the local index of $D$ at any $v\in \Ram(D)$ divides $d$, so 
the assumption that $d$ is coprime to $\deg(v)$ for all $v\in \Ram(D)$ 
is a stronger assumption than (ii) in Corollary \ref{corLargeAut}; it is equivalent to 
assuming that $D_v$, $v\in \Ram(D)$, is a division algebra. 

Via our fixed embedding $D^\times\to \GL_d(\Fi)$, we get an action of $\G:=O_D^\times$ on $\Omega^d$. 
The stabilizer $\G_z$ in $\G$ of a point $z\in \Omega^d$ is a finite subgroup. In fact, Lemma \ref{lem4.5} implies 
that $\G_z$ is the automorphism group of the Drinfeld-Stuhler module corresponding to the lattice $\La_{(z, 1)}$. 
Hence $\G_z\cong \F_{q^s}^\times$ for some $s\mid d$. We say that $z$ is \textit{special} if $\G_z\cong \F_{q^d}^\times$. 
It is clear that for a special point $z\in \Omega$ and any $\gamma\in \G$, the point  
$\gamma(z)$ is also special since $\G_{\gamma(z)}=\gamma\G_z\gamma^{-1}$. 
Denote the set of $\G$-orbits of special points by $\cS(\G)$. 
Denote the set of conjugacy classes of subgroups of $\G$ isomorphic to $\F_{q^d}^\times$ by $\cG(\G)$. 

For a subgroup $G\cong \F_{q^d}^\times$ of $\G$, let $\Fix(G):=\{z\in \Omega\mid gz=z\text{ for all }g\in G\}$. 
If we fix a generator $g$ of $G$, then $\Fix(G)$ coincides with the set of fixed points of $g$. 
The characteristic polynomial of $g$ is separable and irreducible over $\Fi$ (in fact it has coefficients in $\F_q$), so 
$\#\Fix(G)=d$; cf. Lemma \ref{lemFixPts}. 
Let $G_1$ and $G_2$ be two subgroups of $\G$ isomorphic to $\F_{q^d}^\times$. If $G_2=\gamma G_1\gamma^{-1}$ 
for some $\gamma\in \G$, then clearly $\Fix(G_2)=\gamma\Fix(G_1)$. Conversely, suppose $\Fix(G_2)=\gamma\Fix(G_1)$. 
Since $G_i=\G_z$ for any $z\in \Fix(G_i)$ ($i=1,2$) and $\G_{\gamma(z)}=\gamma\G_z\gamma^{-1}$, we see 
that $G_1$ and $G_2$ are conjugate.  
We claim that two distinct points in $\Fix(G)$ are not in the same $\G$-orbit. If this is not the case, then 
there is $\gamma\not \in G$ such that $\gamma G\gamma^{-1}=G$. This contradicts the statement 
at the end of the proof of \cite[Prop. 3.11]{PapikianPAMS}. (This uses the assumption that there is $v\in \Ram(D)$ 
for which $D_v$ is a division algebra.)
Overall, we see that the map $\cS(\G)\to \cG(\G)$, $z\mapsto \G_z$, is well-defined and $d$-to-$1$. 
 
When $A=\F_q[T]$, the number of isomorphism classes of Drinfeld-Stuhler modules over $\C_\infty$ with 
$\Aut(\phi)\cong \F_{q^d}^\times$ is equal to $\# \cS(\G)$, since in this case the double coset set of Proposition \ref{prop1.12} 
is in natural bijection with $\G\bs \Omega^d$.  Thus, 
we need to show that $\# \cG(\G)=d^{\# \Ram(D)-1}$. This last equality is proved in \cite[Prop. 3.11]{PapikianPAMS} 
under the assumptions of the proposition. 
 \end{proof}
 
 \begin{rem}
 Let $\cB^d$ be the geometric realization of the Bruhat-Tits building of $\mathrm{PGL}_d(\Fi)$; see \cite[Ch. 3]{DH}.  
Let $\la: \Omega^d \to \cB^d$ be the $\GL_d(\Fi)$-equivariant map defined in \cite[pp. 64-65]{DH}. 
One can show that the special points in $\Fix(G)$ are conjugate under a natural action of $\Gal(\F_{q^d}/\F_q)$ and 
map to the same vertex of $\cB^d$ under $\la$. The action of $\G$ on $\cB^d$ is studied in \cite{PapikianPAMS}. 
 \end{rem}

\section{Supersingularity}\label{sSS} 

The results on this section are not new. We essentially rephrase some of the results in \cite{LRS} and \cite{PapikianMZ} 
for $\sD$-elliptic sheaves in terms of Drinfeld-Stuhler modules. This reformulation might be 
useful for purposes of future reference. Moreover, Examples \ref{example23} and \ref{exampleEndExceptional} are quite 
instructive, since in special cases they deduce the main result about the 
endomorphism rings of supersingular Drinfeld-Stuhler modules by a direct calculation. 

In this section we fix a maximal ideal $\fp\lhd A$. 
Let $L$ be a field extension of $\F_\fp$ of degree $m$, so $L$ 
is a finite field of order $q^n$, where $n=m\cdot \deg(\fp)$.  Let $\pi=\tau^n$ be 
the associated Frobenius morphism. With abuse of notation, denote by $\pi$ also the 
diagonal matrix $\diag(\pi, \dots, \pi)\in M_d(L[\tau])$. Note that $\pi$ is in the center 
of $M_d(L[\tau])$ since $\tau^n\ell=\ell\tau^n$ for all $\ell\in L$. 
We assume that the $A$-field structure $\gamma: A\to L$ 
factors through the quotient morphism $A\to A/\fp$; in particular, $\chr_A(L)=\fp$. 

As we will see, the theory of Drinfeld-Stuhler modules over $L$ differs considerably depending on whether 
$D$ ramifies at $\fp$ or not (note that this difference already appeared in Lemma \ref{lemFD}). 

\subsection{Case 1: $\fp\not\in\Ram(D)$}

\begin{thm}\label{thm28} 
Let $\phi$ be a Drinfeld-Stuhler $O_D$-module 
defined over $L$. Since $\pi$ commutes with $\phi(O_D)$, we have $\pi\in \End_L(\phi)$. 
Let $\tF:=F(\pi)$ be the subfield of $D':=\End_L(\phi)\otimes_A F$ generated over $F$ by $\pi$. Then:
\begin{enumerate}
\item $[\tF:F]$ divides $d$, and $\infty$ does not split in $\tF/F$. 
\item Let $\widetilde{\infty}$ be the unique place of $\tF$ over $\infty$. 
There is a unique prime $\tilde{\fp}\neq \widetilde{\infty}$ of $\tF$ that divides $\pi$. Moreover, $\tilde{\fp}$ lies above $\fp$.  
\item $D'$ is a central division algebra over $\tF$ of dimension $(d/[\tF:F])^2$ and with invariants 
$$
\inv_{\tilde{v}}(D')=\begin{cases}
-[\tF:F]/d & \text{if $\tilde{v}=\widetilde{\infty}$,}\\
[\tF:F]/d & \text{if $\tilde{v}=\tilde{\fp}$,} \\ 
-[\tF_{\tilde{v}}:F_v]\cdot \inv_v(D) & \text{otherwise,}
\end{cases}
$$
for each place $v$ of $F$ and each place $\tilde{v}$ of $\tF$ dividing $v$. 
\end{enumerate}
\end{thm}
\begin{proof}
Observe that $D'\cong \End(M(\phi)\otimes_A F)^\opp$. 
The theorem then follows from \cite[(9.10)]{LRS} and the equivalences of Section \ref{sMotSht}. 
(We should mention that in Section 9 of \cite{LRS} the $\sD$-elliptic sheaves 
are considered over the algebraic closure of $\F_\fp$. On the other hand, the arguments 
in that section apply also over $L$ with our choice of $(\tF, \pi)$ in place of a ``$\varphi$-pair'' in \cite{LRS}, 
since Theorem A.6 in \cite{LRS} can be proved for $(\tF, \pi)$ as in \cite[$\S$2.2]{LaumonCDV}.) 
\end{proof}

\begin{example}\label{example23} Let $A=\F_q[T]$. 
Let $L$ be the degree $d$ extension of $A/TA\cong \F_q$. 
Let $\phi$ be the Drinfeld-Stuhler $O_D$-module from Example \ref{example12}. 
Assume $(T)$ does not divide $\fr:=\fr(D)$. Fix a generator $h$ of $\F_{q^d}$ over $\F_q$.  
Our Drinfeld-Stuhler module $\phi$ is generated over $\F_q$ 
by 
$\phi_T=\diag(\tau^d, \dots, \tau^d)=\pi$, $\phi_h$, and $\phi_z$,  
which satisfy the relations 
$$
\phi_T\phi_h=\phi_h\phi_T, \quad \phi_T\phi_z=\phi_z\phi_T, \quad \phi_z\phi_h=\phi_{h^q}\phi_z, \quad \phi_z^d=\phi_\fr=\diag(\Phi_\fr, \dots, \Phi_\fr). 
$$  

With abuse of notation, for $i\geq 1$ let
$$
\tau^{i}:=\diag(\tau^{i}, \dots, \tau^{i})\qquad \text{and}\qquad h:=\diag(h, \dots, h). 
$$
Define 
$$
\kappa_i=\phi_z^i \tau^{d-i}, \quad 1\leq i\leq d-1. 
$$
Note that, since the image of $\Phi$ is in $\F_q[\tau]$, we have 
\begin{equation}\label{eqCh5ztau}
\phi_z^i\tau^{d-i}=\tau^{d-i}\phi_z^i. 
\end{equation}
In particular, $h$ and $\kappa_i$ commute 
with $\phi_z$. It is clear that these elements also commute with $\phi_T=\tau^d$. Finally,  
$h$ obviously commutes with $\phi_h$, and so does $\kappa_i$: 
$$
\kappa_i \phi_h=\phi_z^i\tau^{d-i}\phi_h=\phi_z^i\phi_{h^{q^{d-i}}}\tau^{d-i}=\phi_{h^{q^{d}}}\phi_z^i\tau^{d-i}
=\phi_h\kappa_i. 
$$  
We conclude that $E:=\F_q[\phi_T, h, \kappa_1, \dots, \kappa_{d-1}]\subseteq \End_L(\phi)$. 

Note that $h$ and $\kappa_i$ do not commute, 
\begin{equation}\label{eq5.1}
\kappa_i h= h^{q^{d-i}}\kappa_i =\sigma^{-i}(h)\kappa_i, 
\end{equation}
where $\sigma$ is the Frobenius automorphism in $\Gal(\F_{q^d}/\F_q)$. Let $E_i:=\F_q[\phi_T, h, \kappa_i]\subset E$. 
Since $\F_q[\phi_T, h]\cong \F_{q^d}[T]$, we have $E_i=O_K[\kappa_i]$, where $K=\F_{q^d}(T)$. 
Denote $D_i=E_i\otimes_A F$. Combining the relation \eqref{eq5.1} with 
$$
\kappa_i^d = \left(\phi_z^i \tau^{d-i}\right)^d \overset{\eqref{eqCh5ztau}}{=}   \phi_z^{di}  \tau^{d(d-i)} = \phi_{\fr}^i \phi_T^{d-1} = \phi_{\fr^i T^{d-i}}, 
$$
we see that 
for $i$ coprime to $d$ we have 
$$
D_i\cong (K/F, \sigma^{-i}, \fr^i T^{d-i})
$$
(see \eqref{eq-CAP} for the notation). By \cite[(30.4)]{Reiner}, 
for $i$ coprime to $d$, we have $$(K/F, \sigma^{-i}, \fr^i T^{d-i})\cong (K/F, \sigma, \fr^{-1} T).$$ 
Hence for $1\leq i, i'\leq d-1$ coprime to $d$ we have $D_i\cong D_{i'}$, and we denote this cyclic algebra by $\overline{D}$. 
The invariants of $\overline{D}$ are easy to compute using \eqref{eqGossRosen}: 
$$
\inv_v(\overline{D})=\begin{cases}
1/d & \text{if $v=(T)$}, \\ 
-1/d & \text{if $v=\infty$},\\ 
-\inv_v(D) & \text{otherwise}. 
\end{cases}
$$

Let $D':=\End_L(\phi)\otimes_A F$. By Theorem \ref{thmEnd}, we have $\dim_F(D')\leq d^2$. Since $\dim_F \overline{D}=d^2$, 
we conclude that $D'\cong \overline{D}$. Note that the invariants of $\overline{D}$ agree with 
the invariants of $D'$ given by Theorem \ref{thm28}, since in this case $\pi\in F$. 

Next, we claim that $\End_L(\phi)$ is a maximal $A$-order in $D'$. 
One can argue as follows: The discriminant of $E_{1}\subset \End_L(\phi)$ 
is $(\fr T^{d-1})^{d(d-1)}$ (cf. Example \ref{example9}), so $\End_L(\phi)\otimes_A A_\fp$ is a maximal 
$A_\fp$-order in $\overline{D}_\fp$ for all $\fp\neq (T)$. 
On the other hand, the discriminant of $E_{d-1}$ is $(\fr^{d-1}T)^{d(d-1)}$, so $\End_L(\phi)\otimes_A A_T$ 
is a maximal $A_T$-order in $\overline{D}_T$. Since an $A$-order in $\overline{D}$ is maximal 
if and only if it is locally maximal at all 
primes $\fp\lhd A$ (see \cite[(11.6)]{Reiner}), we conclude that $\End_L(\phi)$ is a maximal order.  

Finally, note that $\F_{q^d}^\times\cong \Aut_L(\phi)$. Indeed, $ \F_{q^d}^\times\cong \F_q(h)^\times\subseteq \Aut_L(\phi)$, 
so the equality holds by part (4) of Theorem \ref{thmEnd}. 

This example shows that the bounds on the rank of $\End_L(\phi)$ and the order of $\Aut_L(\phi)$ given 
by Theorem \ref{thmEnd} cannot be improved for fields with non-zero $A$-characteristic. 
\end{example}

\begin{prop}\label{propSS} Let $\phi$ be a Drinfeld-Stuhler $O_D$-module over $L$. The 
following are equivalent: 
\begin{enumerate}
\item $\dim_F(\End(\phi)\otimes_A F)=d^2$;  
\item some power of $\pi$ lies in $A$;
\item there is a unique prime $\tilde{\fp}$ in $\tF$ lying over $\fp$;
\item $\phi[\fp]$ is connected.  
\end{enumerate}
\end{prop}
\begin{proof}
Let $L'$ be a finite extension of $L$ of degree $c$.  
The Frobenius of $L'$ is $\pi^c$. Applying Theorem \ref{thm28}, we see that $\dim_F (\End_{L'}(\phi)\otimes_A F)=d^2$ 
is equivalent to $F(\pi^c)= F$, and since $\pi$ is integral over $A$, this last condition is equivalent to $\pi^c\in A$. 
This shows that (1) and (2) are equivalent. 

Assume (2), i.e., $\pi^c\in A$ for some $c\geq 1$. By 
Theorem \ref{thm28}, $\ord_\fp(\pi^c)\neq 0$. This implies $\ord_\fP(\pi^c)\neq 0$ for any prime $\fP$ in $\tF$ 
lying over $\fp$, and hence also $\ord_\fP(\pi)\neq 0$. Applying Theorem \ref{thm28} again, we conclude that $\fP=\tilde{\fp}$ 
is unique, which is (3). To prove (3)$\Rightarrow$(2), let $f=\Nr_{\tF/F}(\pi)$. We have  
$\ord_\fp(f)>0$ and $\ord_{\fp'}(f)=0$ 
for any prime $\fp'\lhd A$ not equal to $\fp$. Let $\ord_{\tilde{\fp}}(\pi)=u$ and $\ord_{\tilde{\fp}}(f)=w$. 
The element $\pi^w/f^u\in \tF$ has no zeros or poles  away from $\widetilde{\infty}$, 
since $\tilde{\fp}$ is the unique prime over $\fp$ by assumption.  
This implies that $\pi^w/f^u$ lies in the algebraic closure $\F$ of $\F_q$ in $\tF$. 
Therefore, $\pi^{w\kappa}=f^{u\kappa}\in A$, where $\kappa=\#\F-1$. 

Assume (2). Then $\pi^c$ generates $\fp^h$ for some $c, h\geq 1$. This implies that $\phi[\fp]$ is connected, since 
$\phi[\fp]\subseteq \phi[\fp^h]=\ker(\pi^c)$, and $\ker(\pi^c)$ is obviously connected. Thus, (2)$\Rightarrow$(4). 
Conversely, assume $\phi[\fp]$ is connected. Then $\phi[\fp^h]$ is connected for all $h\geq 1$. Choose $h$ 
such that $\fp^h=(a)$ is principal. The assumption that $\phi[a]$ is connected is equivalent to the action of $\tau$ 
on $M(\phi)/aM(\phi)$ being nilpotent, i.e., $\tau^r M(\phi)\subset aM(\phi)\subset \fp M(\phi)$ for all large enough integers $r$; 
cf. \cite[Thm. 5.9]{HartlIsog}.  
By \cite[$\S$6]{PapikianMZ}, this last condition implies that $\dim_F(\End(\phi)\otimes_A F)=d^2$. Hence (4)$\Rightarrow$(1).
 \end{proof}

\begin{defn} A Drinfeld-Stuhler $O_D$-module $\phi$ over $\overline{\F}_\fp$ satisfying the 
equivalent conditions of Proposition \ref{propSS} is called \textit{supersingular}.
(In particular, the Drinfeld-Stuhler module $\phi$ in Example \ref{example23} is supersingular.)
\end{defn}

\begin{thm}\label{thmCh5Endss} Let $\phi$ be a supersingular Drinfeld-Stuhler $O_D$-module over $\overline{\F}_\fp$. We have: 
\begin{enumerate}
\item $\End(\phi)$ is a maximal $A$-order in $\End(\phi)\otimes F$; 
\item $\phi$ can be defined over the extension of $\F_\fp$ of degree $d\cdot \#\Pic(A)$; 
\item the number of isomorphism classes of supersingular Drinfeld-Stuhler $O_D$-modules over $\overline{\F}_\fp$ 
is equal to the class number of $\End(\phi)$; 
\item all supersingular Drinfeld-Stuhler $O_D$-modules are isogenous over $\overline{\F}_\fp$ . 
\end{enumerate}
\end{thm}
\begin{proof}
(1) and (3) are proved in \cite[Thm. 6.2]{PapikianMZ}, (2) follows from \cite[$\S$5]{PapikianCrelle1}, 
(4) follows from \cite[(9.13)]{LRS}. 
\end{proof}

Similar to Remark \ref{rem4.2}, when $O_D=M_d(A)$, 
	the statements of Theorem \ref{thmCh5Endss} are equivalent to some well-known facts about the endomorphism 
	rings of supersingular Drinfeld $A$-modules of rank $d$; cf. \cite{GekelerFDM}.   

\subsection{Case 2: $\fp\in\Ram(D)$} \label{ssSScase2}
We will make a stronger assumption that $D$ is not just ramified at $\fp$, but,  in fact,  
$D_\fp:=D\otimes_F F_\fp$ is a division algebra. 

\begin{lem}\label{lemCh5lemAllss}
	If $D_\fp$ is a division algebra, then a Drinfeld-Stuhler $O_D$-module over a field $L$ of $A$-characteristic $\fp$ 
	is necessarily supersingular, i.e., $\phi[\fp]$ is connected. 
\end{lem}
\begin{proof} The proof is essentially the same as of the analogous fact for quaternionic abelian surfaces; cf. \cite[Lem. 4.1]{RibetBM}. 
	
	For each integer $n\geq 0$, let 
	$H_n:=\phi[\fp^n](L^\sep)$. For $n'\geq n$ we have the inclusion $H_n\subset H_{n'}$ compatible with 
	the left $O_D$-module structures. We define the Tate module of $\phi$ at $\fp$ as 
	$$
	T_\fp(\phi)=\Hom_{A_\fp}(F_\fp/A_\fp, \underset{\substack{\To \\ n}}{\lim}\ H_n). 
	$$
	$T_\fp(\phi)$ is a free $A_\fp$-module of rank $\leq d^2$; cf. \cite{Anderson}, \cite{HartlIsog}. 
	
	It is enough to show that $T_\fp(\phi)=0$. 
	Let $V_\fp(\phi)=T_\fp(\phi)\otimes_{A_\fp} F_\fp$. 
	The division algebra $D_\fp$ acts on $V_\fp(\phi)$. If $V_\fp(\phi)\neq 0$, then for $0\neq v\in V_\fp(\phi)$  
	the $F_\fp$-vector subspace $D_\fp v$ has to have dimension $d^2$, as $xv=0$ implies $x^{-1}(xv)=1v=v=0$, contrary to the assumption. 
	On the other hand, for any $a\in \fp$, since $\partial \phi_a=0$, we have $\phi_a\in M_d(L[\tau])\tau$. 
	Hence $\phi[\fp]$ is not a reduced scheme, which forces $\dim V_\fp(\phi) < d^2$. 
\end{proof}

The number of isomorphism classes of supersingular Drinfeld-Stuhler $O_D$-module over $\overline{\F}_\fp$
is not finite, unlike the case when $\fp\not \in \Ram(D)$. To get finiteness results, and also for certain problems involving 
moduli spaces, one imposes further restrictions on Drinfeld-Stuhler modules in terms of the action of 
$O_D$ on the tangent space. 
For the rest of this subsection we assume that $\inv_\fp(D)=1/d$. 

Let $\phi$ be a Drinfeld-Stuhler $O_D$-module over $k:=\overline{\F}_\fp$. 
Note that $\partial_\phi: O_D\to M_d(k)$ factors through $O_D/\fp$. 
Denote by $\F_\fp^{(d)}$ the degree $d$ extension of $\F_\fp$, and $|\fp|:=\#\F_\fp^{(d)}$. 
It is easy to deduce from \cite[$\S$14]{Reiner} or \cite[Appendix A.2]{LaumonCDV} that 
\begin{align*}
O_D/\fp \cong & \F_\fp^{(d)}\oplus \F_\fp^{(d)}\Pi\oplus \cdots\oplus \F_\fp^{(d)}\Pi^{d-1},\\
& \Pi^d=0, \quad \Pi\alpha=\alpha^{|\fp|}\Pi. 
\end{align*}
Note that $\Pi$ generates a two-sided ideal of $O_D/\fp$, so, in particular, $O_D/\fp \not \cong M_d(\F_\fp)$ as was mentioned in Remark \ref{remNotMd}.
Fix an embedding $\F_\fp^{(d)}\hookrightarrow k$. Via $\partial_\phi$, the submodule $\F_\fp^{(d)}$ of $O_D/\fp$ 
acts on a $d$-dimensional $k$-vector space $V$. 
It is easy to see that $V$ canonically decomposes into a direct sum 
$V=V_1\oplus \cdots \oplus V_d$, where 
$$
V_i:=\left\{v\in V\mid \partial_\phi(\alpha) v = \alpha^{|\fp|^i} v\text{ for all }\alpha\in  \F_\fp^{(d)}\right \}, \quad 0\leq i\leq d-1. 
$$
The following definition is motivated by \cite{RibetBM}, \cite{DrinfeldSym}, and \cite{Hausberger}. 
\begin{defn}
	The \textit{type} of $\phi$ is the ordered $d$-tuple $\mathbf{t}(\phi)=(\dim_k(V_0), \dots, \dim_k(V_{d-1}))$ of non-negative integers. 
	We say that 
	\begin{itemize}
		\item $\phi$ is \textit{special} if $\mathbf{t}(\phi)=(1,1,\dots, 1)$.   
		\item $\phi$ is \textit{exceptional} if $\partial_\phi(\Pi)=0$. 
	\end{itemize}

\end{defn}

\begin{example}\label{exampleExceptional}
	Let $A=\F_q[T]$, $d=2$, and $\fr(D)=T(T-1)$. Let $D$ and $O_D$ be as in Example \ref{example12}. 
	Thus, $D$ be the unique quaternion algebra over $F$ ramified at $T$ and $T-1$ only, and 
	$$
	O_D=\F_{q^2}[T, z]/(z^2-T(T-1)),
	$$
	$$
	zT=Tz, \quad T\alpha=\alpha T,\quad z\alpha=\alpha^q z \text{ for }\alpha\in \F_{q^2}. 
	$$
	Let $\fp=T$. Then $O_D/\fp=\F_{q^2}[z]/z^2$, with $z\alpha=\alpha^q z$ for $\alpha\in \F_\fp^{(2)}\cong \F_{q^2}$. Hence 
	$z$ plays the role of $\Pi$. 
	
	Define $\phi: O_D\to M_2(k[\tau])$ by 
	$$
	\phi_T=\begin{pmatrix} \tau^2 & 0 \\ 0 & \tau^2\end{pmatrix}, \qquad \phi_z=\begin{pmatrix} 0 & 1 \\ \tau^2(\tau^2-1) & 0\end{pmatrix}, 
	\qquad \phi_\alpha=\begin{pmatrix} \alpha & 0 \\ 0 & \alpha^q\end{pmatrix}, \alpha\in \F_{q^2}.
	$$
	It is easy to check that $\phi$ is a Drinfeld-Stuhler $O_D$-module which is special but not exceptional.
	
	Now define $\phi: O_D\to M_2(k[\tau])$ by
	$$
	\phi_T=\begin{pmatrix} \tau^2 & 0 \\ 0 & \tau^2\end{pmatrix}, \qquad \phi_z=\begin{pmatrix} 0 & \tau \\  \tau(\tau^2-1) & 0\end{pmatrix}, 
	\qquad \phi_\alpha=\begin{pmatrix} \alpha & 0 \\ 0 & \alpha\end{pmatrix}, \alpha\in \F_{q^2}. 
	$$
	This module is exceptional but not special. Its type is $(2,0)$. 
	
	Finally, define $\phi: O_D\to M_2(k[\tau])$ by
	$$
	\phi_T=\begin{pmatrix} \tau^2 & 0 \\ 0 & \tau^2\end{pmatrix}, \qquad \phi_z=\begin{pmatrix} 0 & \tau^2 \\  \tau^2-1 & 0\end{pmatrix}, 
	\qquad \phi_\alpha=\begin{pmatrix} \alpha & 0 \\ 0 & \alpha^q\end{pmatrix}, \alpha\in \F_{q^2}. 
	$$
	This module is exceptional and special (such modules are called \textit{superspecial}). 
	
	We remark that, for $d=2$, if $\phi$ is not exceptional then it is special. Indeed, 
	the commutation relation $\Pi\alpha=\alpha^{|\fp|}\Pi$ implies that $\partial_\phi(\Pi)V_0\subset V_1$ and vice versa. 
	Hence, if $\partial_\phi(\Pi)\neq 0$, then $V_0\neq V$ and $V_1\neq V$, so both must be $1$-dimensional. 
\end{example}

Special Drinfeld-Stuhler modules, or rather their $\sD$-elliptic sheaf counterparts, play an important role in the 
construction of moduli schemes of $\sD$-elliptic sheaves over $\Spec(A_\fp)$ with nice geometric properties; see \cite{Hausberger}. 
On the other hand, endomorphism rings of exceptional  Drinfeld-Stuhler modules have properties similar to 
the properties of supersingular modules in Theorem \ref{thmCh5Endss}: 

\begin{thm}
	Let $\phi$ be an exceptional Drinfeld-Stuhler $O_D$-module over $k$ of type $\mathbf{t}$. Then 
	$\End(\phi)$ is a hereditary $A$-order in the central division algebra $\overline{D}$ with invariants 
	$$
	\inv_v(\overline{D}) = 
	\begin{cases}
	- 1/d & \text{if } v=\infty, \\ 
	0 & \text{if } v=\fp, \\
	- \inv_v(D) & \text{if } v\neq \fp, \infty. 
	\end{cases}
	$$
	This order is maximal at every finite place $v\neq \fp$, and at $\fp$ it is isomorphic to a hereditary order 
	uniquely determined by $\mathbf{t}$. Moreover, the number of isomorphism classes of 
	exceptional Drinfeld-Stuhler $O_D$-module over $k$ of type $\mathbf{t}$ is equal to the class number of $\End(\phi)$. 
\end{thm}
\begin{proof}
	See \cite{PapikianMZ}, where the reader  will also find the precise relationship between the type $\mathbf{t}$ and the hereditary order at $\fp$. 
	For example, $\End(\phi)$ is maximal at $\fp$ if and only if, up to a cyclic permutation, $\mathbf{t}=(d, 0,\dots, 0)$. 
\end{proof}

\begin{example}\label{exampleEndExceptional} Let $\phi$ be the exceptional Drinfeld-Stuhler module over $k$ of type $(2, 0)$ from Example \ref{exampleExceptional}: 
	$$
	\phi_T=\begin{pmatrix} \tau^2 & 0 \\ 0 & \tau^2\end{pmatrix}, \qquad \phi_z=\begin{pmatrix} 0 & \tau \\  \tau(\tau^2-1) & 0\end{pmatrix}, 
	\qquad \phi_\alpha=\begin{pmatrix} \alpha & 0 \\ 0 & \alpha\end{pmatrix}, \alpha\in \F_{q^2}. 
	$$
	Fix a generator $\beta$ of $\F_T^{(2)}\cong \F_{q^2}$ over $\F_q$,  and let 
	$$
	h:=\begin{pmatrix} \beta & 0 \\ 0 & \beta^q\end{pmatrix}, \qquad \kappa:=\phi_z\tau=\tau\phi_z= \begin{pmatrix} 0 & \tau^2 \\  \tau^2(\tau^2-1) & 0\end{pmatrix}. 
	$$
	By a straightforward calculation one checks that $E:=\F_q[\phi_T, h, \kappa]\subseteq \End(\phi)$ and 
	$$
	\kappa h= h^q\kappa, \qquad \kappa^2=\phi_{T^2(T-1)}. 
	$$
	It is easy to see from this that $\overline{D}:=E\otimes_A F$ is the quaternion algebra over $F$ ramified at $(T-1)$ and $\infty$. 
	Moreover, as in Example \ref{example23}, one can check that $\End(\phi)$ is a maximal $A$-order in $\overline{D}$. 
\end{example}


\section{Fields of moduli}\label{sFM} 

The main results of this section are about the fields of moduli of Drinfeld-Stuhler modules. 
As an auxiliary tool, we will need a Hilbert's 90-th type theorem for $\GL_d(L^\sep[\tau])$, which we prove first.  
Throughout the section, $L$ is an arbitrary $A$-field with $\chr_A(L)\nmid \fr(D)$.  


\begin{lem}\label{lem6.1} We have: 
\begin{enumerate}
\item Every left ideal of $L[\tau]$ is principal. 
\item Every finitely generated torsion-free left $L[\tau]$-module is free. 
\end{enumerate}
\end{lem}
\begin{proof}
(1) follows from the existence of the right division algorithm for $L[\tau]$ (see \cite[Cor. 1.6.3]{Goss}), and (2) essentially 
follows from the same fact (see \cite[Cor. 5.4.9]{Goss}). 
\end{proof}

Let $K$ be a finite Galois extension of $L$ of degree $n$. Let $\sigma_1, \sigma_2, \dots, \sigma_n$ be the 
elements of $G:=\Gal(K/L)$. 
The Galois group $G$ 
acts on $K[\tau]$ via the obvious action on the coefficients of polynomials, and it acts on the ring 
$M_d(K[\tau])$ by acting on the entries of matrices. 
Let $M:=K[\tau]^d$ be the free left $K[\tau]$-module of rank $d$. 
Then $\GL_d(K[\tau])$ can be identified with the group $\Aut_{K[\tau]}(M)$ of automorphism of $M$, where 
$g\in \GL_d(K[\tau])$ acts on $M$ from the right as on row vectors. (Of course, $\GL_d(K[\tau])$ also 
acts on $M$ from the left as on column vectors, but that action is not $K[\tau]$-linear.) From this identification 
it is easy to see the validity of the following:

\begin{lem}\label{lemLinInd}
If $v_1, \dots, v_d \in M$ form a left $K[\tau]$-basis of $M$, then the matrix $S$ whose rows are $v_1, \dots, v_d$ is in $\GL_d(K[\tau])$. 
Conversely, the rows of $S\in \GL_d(K[\tau])$ form a left $K[\tau]$-basis of $M$. 
\end{lem}

\begin{rem} 
There are matrices in $M_d(K[\tau])$ whose rows are left linearly independent but whose 
columns are left linearly dependent over $K[\tau]$, e.g., 
$\begin{pmatrix} 1 & \tau \\ 
\alpha+\tau & \tau(\alpha+\tau)
\end{pmatrix}$ where $\alpha\in K$ is such that $\alpha^q\neq \alpha$. 
\end{rem}

\begin{lem}\label{lem6.2} The inclusion $L[\tau]\subset K[\tau]$ 
makes $K[\tau]$ into a left $L[\tau]$-module. As such, 
$K[\tau]$ is a free left $L[\tau]$-module of rank $n$. 
\end{lem}
\begin{proof}
It is obvious that $K[\tau]$ has no torsion elements for the action of $L[\tau]$. Let $\alpha_1, \dots, \alpha_n\in K$ 
be an $L$-basis of $K$. It is enough to show that $K[\tau]=\sum_{i=1}^n L[\tau]\alpha_i$. 
By the Dedekind's theorem on the independence of characters, $\{\alpha_1, \dots, \alpha_n\}$ form an $L$-basis of $K$ 
if and only if $\det(\sigma_i\alpha_j)\neq 0$. On the other hand, $\det(\sigma_i\alpha_j)\neq 0$ 
if and only if $\det(\sigma_i\alpha_j)^{q^r}=\det(\sigma_i\alpha_j^{q^r})\neq 0$ for any $r\geq 0$. Hence 
$\{\alpha_1^{q^r}, \dots, \alpha_n^{q^r}\}$ is also an $L$-basis of $K$. 
Let $f=a_0+a_1\tau+\cdots + a_k \tau^k\in K[\tau]$. 
For each $a_i$ we can find $b_{i,1}, \dots, b_{i, n}\in L$ such that $\sum_{j=1}^n b_{i, j} \alpha_j^{q^i}=a_i$. 
Thus, $f=\sum_{j=1}^n g_j \alpha_j$, where $g_j:=\sum_{i=1}^k b_{i, j}\tau^i\in L[\tau]$.  
\end{proof}
 
\begin{defn}
 We say that $G$ acts on $M$ by \textit{semi-linear automorphisms} (cf. \cite[p. 110]{Berhuy}), 
 $G\times M\to M$, $(\sigma, m)\mapsto \sigma\ast m$, if for all $m, m'\in M$, $\sigma\in G$,  
 and $\la\in K[\tau]$ we have 
 \begin{enumerate}
 \item[(i)] $\sigma\ast (m+m')=\sigma\ast m + \sigma\ast m'$,
 \item[(ii)] $\sigma\ast (\la m)=\sigma\la\cdot \sigma\ast m$,
 \end{enumerate}
 where $\sigma\la$ denotes the usual action of $G$ on $K[\tau]$, and the dot denotes the action of $K[\tau]$ on $M$. Let 
 $$
 M^G:=\{m\in M\mid \sigma\ast m=m\textit{ for all }\sigma\in G\}. 
 $$
 It is easy to see that $M^G$ is a left $L[\tau]$-module. 
 \end{defn}
 
\begin{lem}\label{lem6.3} The left $L[\tau]$-module $M^G$ is free of rank $d$, i.e., $M^G\cong L[\tau]^d$. 
Moreover, the map $K\otimes_L M^G\to M$, $\alpha\otimes m\mapsto \alpha m$, 
is an isomorphism. 
\end{lem}
\begin{proof} 
Since every left ideal of $L[\tau]$ is principal (Lemma \ref{lem6.1}), every submodule 
of a free left $L[\tau]$-module of finite rank is also free of finite rank (cf. \cite[(2.44)]{Reiner}). 
On the other hand, by Lemma \ref{lem6.2}, the left $L[\tau]$-module $M$ is free of finite rank. Hence 
the $L[\tau]$-submodule $M^G$ of $M$ is also free of finite rank. To show that the rank of $M^G$ 
over $L[\tau]$ is $d$, it is enough to show that the map $K\otimes_L M^G\to M$, $\alpha\otimes m\mapsto \alpha m$, 
is an isomorphism. This last isomorphism follows from the Galois descent for vector spaces; see \cite[Lem. III.8.21]{Berhuy}. 
\end{proof}

\begin{prop}\label{propHilb90}
Let $c:G\to \GL_d(K[\tau])$, $\sigma\mapsto c_\sigma$, be a map which satisfies 
$c_{\sigma\delta}=\sigma(c_\delta) c_\sigma$ for all $\sigma, \delta\in G$. Then there 
is a matrix $S\in \GL_d(K[\tau])$ such that $c_\sigma=(\sigma S)^{-1}S$ for all $\sigma\in G$. 
\end{prop}
\begin{proof} 
Define a (twisted) action of $G$ on $M$: 
$$
\sigma\ast m=(\sigma m)c_\sigma \quad \text{for all }m\in M, \sigma\in G. 
$$
One easily checks that $(\sigma\delta)\ast m=\sigma\ast (\delta\ast m)$ for all $\sigma,\delta \in G$ and $m\in M$, so 
this is indeed an action. Moreover, this action is semi-linear.   
Using Lemma \ref{lem6.3}, we can choose a basis $v_1, \dots, v_d$ of the left $L[\tau]$-module $M^G\cong L[\tau]^d$ 
such that $\sum_{i=1}^d K[\tau] v_i = M$. Since $M$ is a free left $K[\tau]$-module of rank $d$, 
the elements $v_1, \dots, v_d$ form a left $K[\tau]$-basis of $M$. Let $S$ be the matrix whose rows are $v_1, \dots, v_d$. 
By Lemma \ref{lemLinInd}, $S\in \GL_d(K[\tau])$. The relations 
$$
v_i=\sigma\ast v_i=(\sigma v_i) c_\sigma \text{ for all }i=1, \dots, d,
$$
are equivalent to the matrix equality $S=(\sigma S) c_\sigma$ for all $\sigma\in G$, and this implies the claim of the lemma. 
\end{proof}

\begin{lem}\label{lem6.7} Let $\phi$ be a Drinfeld-Stuhler $O_D$-module over $K$ with 
$\Aut_K(\phi)\cong \F_{q^r}^\times$ $($cf. Theorem \ref{thmEnd}$)$. 
Then 
$$
\partial: \Aut_K(\phi)\to \GL_d(K)
$$
gives an isomorphism from the group $\Aut_K(\phi)$ to the group $\cA:=\{\diag(\alpha, \dots, \alpha)\mid \alpha\in \F_{q^r}^\times\}$.  
\end{lem}
\begin{proof}
By Lemma \ref{lemFD}, $\partial(\Aut_K(\phi))$ lies in the center of $\GL_d(K)$. Since the center of 
$\GL_d(K)$ consists of scalar matrices, and the $(q^r-1)$-th roots of $1$ in $K$ are the elements of $\F_{q^r}^\times$, 
the restriction of $\partial$ to $\Aut_K(\phi)$ is indeed a homomorphism into 
$\cA$. Since $\Aut_K(\phi)\cong \cA$, to prove that  
$\partial$ is an isomorphism it is enough to prove that it is injective. Let $h:=q^r-1$. Assume $\alpha\in \Aut_K(\phi)$ 
is such that $\partial(\alpha)=1$. Then we can write $\alpha=1+\sum_{i=1}^n B_i\tau^n$ for some $n\geq 1$. 
Suppose not all $B_i$ are zero, and let $m$ be the smallest index such that $B_m\neq 0$. Then 
$$
1=\alpha^h=1+hB_m\tau^m+\cdots, 
$$ 
which implies $hB_m=0$. Since $h$ is coprime to the characteristic of $K$, we must have $B_m=0$, which is a 
contradiction. 
\end{proof}

\begin{rem}
It is not generally true that the elements of $\Aut_K(\phi)$ are scalar matrices in $\GL_d(K[\tau])$. For example, 
suppose $d=2$, $\diag(\alpha, \alpha)\in \Aut_K(\phi)$, and $\alpha\not\in \F_q$. 
Let $S=\begin{pmatrix} 1 & \tau\\ 0 & 1\end{pmatrix}\in \GL_2(K[\tau])$. Then 
$\begin{pmatrix} \alpha & (\alpha^q-\alpha)\tau\\ 0 & \alpha\end{pmatrix}\in \Aut_K(\psi)$, where 
$\psi$ is the Drinfeld-Stuhler module $S\phi S^{-1}$.  
\end{rem}

 \begin{defn}
 Let $\phi$ be a Drinfeld-Stuhler $O_D$-module over $L^\sep$. For $\sigma\in \Gal(L^\sep/L)$, let $\phi^\sigma$ 
 be the composition
 $$
 \phi^\sigma: O_D\overset{\phi}{\To} M_d(L^\sep[\tau])\overset{\sigma}{\To} M_d(L^\sep[\tau]). 
 $$
 It is easy to check that $\phi^\sigma$ is again a Drinfeld-Stuhler $O_D$-module. We say that $L$ 
 is a \textit{field of moduli for $\phi$} if for all $\sigma\in \Gal(L^\sep/L)$ the Drinfeld-Stuhler module 
 $\phi^\sigma$ is isomorphic to $\phi$.  
 \end{defn}
 
 If $L$ is a field of definition for $\phi$, then $L$ is obviously a field of moduli. 
 
 \begin{thm}\label{thmFMFD}
 Let $\phi$ be a Drinfeld-Stuhler $O_D$-module over $L^\sep$ with $\Aut(\phi)\cong \F_{q^r}^\times$. 
 Assume $L$ is a field of moduli for $\phi$. If $d$ and $q^r-1$ are coprime, then $L$ is a field of definition for $\phi$. 
 \end{thm}
 \begin{proof} 
We can find a finite Galois extension $K$ of $L$ such that $\phi$ is defined over $K$ 
and all isomorphisms of $\phi$ to $\phi^\sigma$ for every $\sigma\in \Gal(K/L)$ are defined over $K$. 
(Take, for example, $K$ such that $\phi$ and $\phi[a]$ are defined over $K$, where $a\in A$ is coprime with $\chr_A(L)$ and $\fr(D)$.)
In particular, $\Aut_K(\phi)=\Aut(\phi)$. 
Denote $G=\Gal(K/L)$. For each $\sigma\in G$, choose an isomorphism $\la_\sigma: \phi\to \phi^\sigma$. 
Then 
\begin{equation}\label{eqCalc}
\la_{\sigma\delta} \phi\la_{\sigma\delta}^{-1}=\phi^{\sigma\delta}=(\phi^\delta)^\sigma=(\la_\delta \phi\la_\delta^{-1})^\sigma 
=\sigma(\la_\delta)\phi^\sigma \sigma(\la_\delta)^{-1}= \sigma(\la_\delta)\la_\sigma\phi \la_\sigma^{-1} (\sigma \la_\delta)^{-1}. 
\end{equation}
Hence 
\begin{equation}\label{eq-pre-cocycle}
\la_{\sigma\delta} = \sigma(\la_\delta)\la_\sigma \alpha_{\sigma, \delta}
\end{equation}
with $\alpha_{\sigma, \delta}\in \Aut(\phi)$. 

Let $\underline{\det}:\GL_d(K[\tau])\to K^\times$ be the composition 
 $$
\underline{\det}: \GL_d(K[\tau])\overset{\partial}{\To} \GL_d(K)\overset{\det}{\To} K^\times. 
 $$
The assumption that $d$ and $q^r-1$ are coprime, combined with Lemma \ref{lem6.7}, implies that 
$\underline{\det}: \Aut(\phi)\overset{\sim}{\To}\F_{q^r}^\times$ is an isomorphism. 
Denote $\mu_\sigma= \underline{\det}(\la_\sigma)$ and $h=q^r-1$. Then  
$\mu_{\sigma\delta} = \sigma(\mu_\delta)\mu_\sigma \underline{\det}(\alpha_{\sigma, \delta})$, and 
$
\mu_{\sigma\delta}^h = \sigma(\mu_\delta^h)\mu_\sigma^h$. 
Hence $G\to K^\times$, $\sigma\mapsto \mu_\sigma^h$, is a $1$-cocycle. By Hilbert's Theorem 90 
for $K^\times$, there is $b\in K^\times$ such that $\mu_{\sigma}^h=b/\sigma(b)$ for all $\sigma\in G$. 
Let $a$ be an element of $L^\sep$ such that $a^h=b$. The extension $K':=K(a)$ is Galois over $L$. 
Put $G^\ast=\Gal(K'/L)$, and let $\pi: G^\ast\to G$ be the natural homomorphism. For every $\sigma\in G^\ast$, 
we see that $\mu_{\pi(\sigma)}\sigma(a)/a$ is an $h$-th root of unity, hence there is a unique $\alpha_\sigma\in \Aut(\phi)$ 
such that $\mu_{\pi(\sigma)}\underline{\det}(\alpha_\sigma)=a/\sigma(a)$. 

Put $c_\sigma=\la_{\pi(\sigma)}\alpha_\sigma$. 
Then $c_\sigma: \phi\to \phi^\sigma$ is an isomorphism  and $\underline{\det}(c_\sigma)=a/\sigma(a)$. 
Repeating the calculation \eqref{eqCalc} for $c_\sigma$, we 
arrive at the relations $c_{\sigma\delta}=\sigma(c_\delta)c_\sigma\beta_{\sigma, \delta}$ for some $\beta_{\sigma, \delta}\in \Aut(\phi)$. 
But now, taking $\underline{\det}$ of both sides, we have 
$$
\frac{a}{\sigma\delta(a)}=\frac{\sigma(a)}{\sigma(\delta(a))}\frac{a}{\sigma(a)}\underline{\det}(\beta_{\sigma, \delta}). 
$$
Thus $\underline{\det}(\beta_{\sigma, \delta})=1$. Since $\underline{\det}: \Aut(\phi)\to K^\times$ is injective, we 
must have $\beta_{\sigma, \delta}=1$. Therefore, $c_{\sigma\delta}=\sigma(c_\delta)c_\sigma$ for all $\sigma,\delta\in G^\ast$. 
By Proposition \ref{propHilb90}, there is $S\in \GL_d(K'[\tau])$ such that $c_\sigma=(\sigma S)^{-1}S$ for all $\sigma\in G^\ast$. 
Put $\psi=S\phi S^{-1}$; this is a Drinfeld-Stuhler module isomorphic to $\phi$ over $K'$. For any $\sigma\in G^\ast$ we have 
$$
 \psi^\sigma = (S\phi S^{-1})^\sigma=(\sigma S)(\phi^\sigma )(\sigma S^{-1}) 
 =  (\sigma S)(c_\sigma \phi c_\sigma^{-1})(\sigma S)^{-1} = S \phi S^{-1}=\psi,
$$
so $\psi$ is defined over $L$. 
 \end{proof}
 
 \begin{rem}\label{rem611}
 Recall from Theorem \ref{thmEnd} that $\Aut(\phi)\cong \F_{q^s}^\times$ for some $s$ dividing $d$. Therefore, 
 the assumption in Theorem \ref{thmFMFD} can be replaced by a universal but stronger assumption that 
 $d$ and $q^d-1$ are coprime. Note that if $d=p^e$ is a power of the characteristic of $F$, then 
 the assumption of Theorem \ref{thmFMFD} is always satisfied. On the other hand, if $d=\ell$ is a prime 
 different from $p$, then the assumption is satisfied if and only if $\ell$ does not divide $q-1$. 
 \end{rem}

 \begin{thm}\label{thmFMFD2}
 Let $\phi$ be a Drinfeld-Stuhler $O_D$-module over $L^\sep$. Assume $L$ is a field of moduli 
 for $\phi$. Then $L$ is a field of definition for $\phi$ if and only if $O_D\otimes_A L\cong M_d(L)$. 
 \end{thm}
 \begin{proof}
 If $L$ is a field of definition for $\phi$, then $O_D\otimes_A L\cong M_d(L)$ by Lemma \ref{lemFD}, so 
 the condition is necessary. 
 
 Now assume $O_D\otimes_A L\cong M_d(L)$. As in the proof of Theorem \ref{thmFMFD}, 
 let $K/L$ be a finite Galois extension such that $\phi$ is defined over $K$ 
and all isomorphisms of $\phi$ to $\phi^\sigma$ for every $\sigma\in \Gal(K/L)=:G$ are defined over $K$. 
We have $\partial_{\phi, K}: O_D\otimes_A K\xrightarrow{\sim} M_d(K)$. The group $G$ acts on 
$O_D\otimes_A K$ via its action on $K$. We fix an element $e\in O_D\otimes_A L $, and 
consider it as an element of $O_D\otimes_A K$ fixed by $G$. 

Denote 
$$
V(\phi, e)=\ker(\partial_{\phi, K}(e))\subset K^d. 
$$
For each $\sigma\in G$ choose an isomorphism $\la_\sigma: \phi\xrightarrow{\sim}\phi^\sigma$, and 
denote $M_\sigma=\partial(\la_\sigma)\in \GL_d(K)$. 
Then 
\begin{equation}\label{eqVphie}
M_\sigma V(\phi, e) = \sigma V(\phi, e),
\end{equation}
where $\sigma V(\phi, e)$ denotes the image of $V(\phi, e)$ under the action of $\sigma$ on $K^d$; we prove \eqref{eqVphie} 
in two steps: 
\begin{itemize}
\item[(i)] Because $\la_\sigma \phi =\phi^\sigma \la_\sigma$, we have 
$M_\sigma \partial_{\phi, K}(e)=\partial_{\phi^\sigma, K}(e)M_\sigma$. Hence the matrices $\partial_{\phi, K}(e)$ and 
$\partial_{\phi^\sigma, K}(e)$ have the same rank, which implies $\dim_K V(\phi,e) = \dim_K V(\phi^\sigma, e)$. 
If $v\in V(\phi, e)$, then 
 $\partial_{\phi, K}(e)(v)=0$, so $M_\sigma(v)\in V(\phi^\sigma, e)$. This implies that $M_\sigma 
 V(\phi,e)\subseteq V(\phi^\sigma, e)$. Since $\dim_K V(\phi,e) = \dim_K V(\phi^\sigma, e)$, we 
 must have $M_\sigma V(\phi,e)=V(\phi^\sigma, e)$.
 \item[(ii)] If $v\in V(\phi, e)$, then 
 $$
 \sigma(\partial_{\phi, K}(e)(v))=\partial_{\phi^\sigma, K}(\sigma(e))(\sigma(v))
 =\partial_{\phi^\sigma, K}(e)(\sigma(v))=0, 
 $$
 where we have used the assumption that $\sigma(e)=e$. Thus $\sigma(v)\in V(\phi^\sigma, e)$, which implies 
 $\sigma V(\phi, e) = V(\phi^\sigma, e)$. 
\end{itemize}

We can choose $e$ such that 
$V(\phi, e)=K\omega$ is one-dimensional, spanned by some $\omega\in K^d$. (Note that the rank of $e$ 
does not change whether we consider it as an element of $O_D\otimes_A L\cong M_d(L)$ or 
$O_D\otimes_A K\cong (O_D\otimes_A L)\otimes_L K\cong M_d(K)$.)
Then \eqref{eqVphie} becomes 
$$
M_\sigma \omega = \mu_\sigma \sigma(\omega) \quad \text{for some}\quad \mu_\sigma\in K^\times.  
$$
Using \eqref{eq-pre-cocycle}, we get 
\begin{equation}\label{eqMsigma}
M_{\sigma\delta}=(\sigma M_\delta) M_\sigma \partial(\alpha_{\sigma, \delta}),
\end{equation}
where $\alpha_{\sigma, \delta}\in \Aut(\phi)$ is the automorphism 
appearing in \eqref{eq-pre-cocycle}. By Lemma \ref{lem6.7}, $\partial(\alpha_{\sigma, \delta})$ is a scalar matrix, so 
applying both sides of \eqref{eqMsigma} to $\omega$, we get an equality in $K^\times$
$$
\mu_{\sigma\delta}=\sigma(\mu_\delta) \mu_\sigma \partial(\alpha_{\sigma, \delta}). 
$$

At this point one can repeat the argument in the proof of Theorem \ref{thmFMFD}, with $\partial(\alpha_{\sigma, \delta})$ 
playing the role of $\underline{\det}(\alpha_{\sigma, \delta})$, to obtain a Drinfeld-Stuhler module $\psi$ 
over $L$, which is isomorphic to $\phi$ over $L^\sep$. 
 \end{proof}

 \begin{rem}
	Theorem \ref{thmFMFD} is the analogue for Drinfeld-Stuhler modules of a theorem of Shimura for abelian varieties \cite[Thm. 9.5]{Shimura}, and  
	Theorem \ref{thmFMFD2} is the analogue of a 
	theorem of Jordan for abelian surfaces with quaternionic multiplication \cite[Thm. 1.1]{Jordan}. 
\end{rem}

\begin{cor}\label{corFMFD-DM}
	Let $\Phi: A\to L^\sep[\tau]$ be a Drinfeld $A$-module of rank $r\geq 1$. If $L$ is a field of moduli for $\Phi$, then $L$ 
	is also a field of definition. 
\end{cor}
\begin{proof}
	The claim follows from Theorem \ref{thmFMFD2} specialized to $O_D=M_r(A)$ and Theorem \ref{thmME-DSM}. 
	Alternatively, one can arrive at the same conclusion by repeating for $\Phi$ the argument in the proof of Theorem \ref{thmFMFD} with $d=1$ 
	(but arbitrary rank), 
	which becomes simpler since in that case Proposition \ref{propHilb90} is just the usual Hilbert's 90.
\end{proof}
\begin{rem}
	Corollary \ref{corFMFD-DM} is the analogue of 
	the well-known fact  that the fields of moduli for elliptic curves are fields of definition; cf. \cite[Prop. I.4.5]{SilvermanII}. 
	The proof for elliptic curves uses the $j$-invariant, an invariant which is not available for Drinfeld modules if $A$ is not 
	the polynomial ring or the rank is greater than $2$.  
\end{rem}

 In \cite{LRS}, $\sD$-elliptic sheaves are defined over any $\F_q$-scheme $S$. The functor 
which associates to $S$ the set of isomorphism classes of $\sD$-elliptic sheaves over $S$ modulo 
the action of $\Z$ possesses a coarse moduli scheme $X^\sD$ which is a smooth proper scheme 
over $C':=C-\mathrm{Ram}(D)-\{\infty\}$ of relative dimension $(d-1)$; 
this follows from Theorems 4.1 and 6.1 in \cite{LRS}, combined with the Keel-Mori theorem. 
Thanks to Theorems \ref{thmModMot} and \ref{thmShtMot}, the fibre of this moduli scheme 
over a (not necessarily closed) point $x$ of $C'$ is the coarse moduli space of isomorphism classes of Drinfeld-Stuhler 
$O_D$-modules over fields $L$ such that $z(\Spec(L))=x$ (recall from Definition \ref{defnDES} that $z$ is the morphism induced by $\gamma$). 

\begin{cor} Let $z(\Spec(L))=x$ and $X^\sD_x$ be the fibre of $X^\sD$ over the point $x$ of $C'$. 
\begin{enumerate}
\item Assume $d$ and $q^d-1$ are coprime. 
If $O_D\otimes_A L\not \cong M_d(L)$, then $X^\sD_x(L)=\emptyset$.
\item Assume $O_D\otimes_A L\cong M_d(L)$ and $X^\sD_x(L)\neq \emptyset$. Then 
there is a Drinfeld-Stuhler $O_D$-module defined over $L$. 
\end{enumerate}
\end{cor}
\begin{proof}
Given $a\in A-\F_q$ coprime with $\fr(D)$, 
one can consider the problem of classifying Drinfeld-Stuhler modules with level-$a$ 
structures, i.e., classifying pairs $(\phi, \iota)$, where $\phi$ 
is a Drinfeld-Stuhler $O_D$-module  and $\iota$ is an isomorphism $\iota: \phi[a]\cong O_D/aO_D$.
 This moduli problem 
is representable; see \cite[(5.1)]{LRS}.  
Denote the corresponding moduli scheme by $X^{\sD, a}$. The forgetful map $(\phi, \iota)\mapsto \phi$ gives 
a Galois covering $X^{\sD, a}_{x}\to X^\sD_x$. 
Suppose there is an $L$-rational point $P$ on $X^\sD_x$. Then a preimage $P'$ of $P$ in $X^{\sD, a}_{x}$ 
is defined over a Galois extension $L'$ of $L$.  Since $X^{\sD, a}_{x}$ is a fine moduli scheme, there 
is a Drinfeld-Stuhler $O_D$-module $\phi$ defined over $L'$ which corresponds to $P'$. 
For any $\sigma\in \Gal(L'/L)$, the Drinfeld-Stuhler modules $\phi$ and $\phi^\sigma$ are isomorphic over $L'$, 
since $\phi$ arises from an $L$-rational point on $X^\sD_x$.  Hence $L$ is a field of moduli of $\phi$. 

If $d$ and $q^d-1$ are coprime, then Theorem \ref{thmFMFD} and Remark \ref{rem611}  imply 
that $\phi$ can be defined over $L$. Now Lemma \ref{lemFD} implies that $O_D\otimes_A L\cong M_d(L)$. This proves part (1). 
Part (2) follows from Theorem \ref{thmFMFD2}. 
\end{proof}
 
\begin{rem}
 It is known that in general the fields of moduli for abelian varieties are not necessarily fields of definition. 
 For example, let $B$ be an indefinite quaternion division algebra over $\Q$, and let $X^B$ be the associated 
 Shimura curve over $\Q$, which is the coarse moduli scheme of abelian surfaces equipped with an action of $B$.  
 The main result in \cite{JLlocal} provides examples of nonarchimedean local fields $L$ failing to split $B$ with $X^B(L)\neq \emptyset$  
 (see also \cite[$\S$1]{Jordan});  a necessary condition for this phenomenon is that $2$ ramifies in $B$. 
 If we let $A=\F_q[T]$, $F=\F_q(T)$, and $d=2$, then $X^\sD_F$ is the function field analogue of $X^B$; cf. \cite{LRS}, \cite{PapLocProp}. 
 However, examples similar to those constructed by Jordan and Livn\'e 
 do not exist in this setting since for any finite extension $L$ of $F_v$, $v\in \Ram(D)$, which does not split $D$ 
 we have $X^{\sD}_F(L)=\emptyset$ by Theorem 4.1 in \cite{PapLocProp}. This leaves open the interesting question 
 whether in general the fields of moduli of Drinfeld-Stuhler modules are fields of definition.  
 \end{rem}
 
\subsection*{Acknowledgements} This work was mostly carried out during 
 my visit to the Max Planck Institute for Mathematics in Bonn in 2016.  
 I thank the institute for its hospitality, excellent working conditions, 
 and financial support. 
 I thank Gebhard B\"ockle, Urs Hartl, Rudolph Perkins, Fu-Tsun Wei, and Yuri Zarhin for helpful discussions related to the 
 topics of this paper.



\begin{thebibliography}{10}
	
	\bibitem{Anderson}
	G.~Anderson, \emph{{$t$}-motives}, Duke Math. J. \textbf{53} (1986), no.~2,
	457--502.
	
	\bibitem{BMcL}
	M.~F. Becker and S.~MacLane, \emph{The minimum number of generators for
		inseparable algebraic extensions}, Bull. Amer. Math. Soc. \textbf{46} (1940),
	182--186.
	
	\bibitem{Berhuy}
	G.~Berhuy, \emph{An introduction to {G}alois cohomology and its applications},
	London Mathematical Society Lecture Note Series, vol. 377.
	
	\bibitem{BS}
	A.~Blum and U.~Stuhler, \emph{Drinfeld modules and elliptic sheaves}, Vector
	bundles on curves---new directions ({C}etraro, 1995), Lecture Notes in Math.,
	vol. 1649, Springer, Berlin, 1997, pp.~110--193.
	
	\bibitem{BG}
	G.~B\"ockle and D.~Gvirtz, \emph{Division algebras and maximal orders for given
		invariants}, LMS J. Comput. Math. \textbf{19} (2016), no.~suppl. A, 178--195.
	
	\bibitem{BH}
	M.~Bornhofen and U.~Hartl, \emph{Pure {A}nderson motives and abelian
		{$\tau$}-sheaves}, Math. Z. \textbf{268} (2011), no.~1-2, 67--100.
	
	\bibitem{DH}
	P.~Deligne and D.~Husem\"oller, \emph{Survey of {D}rinfeld modules}, Current
	trends in arithmetical algebraic geometry ({A}rcata, {C}alif., 1985),
	Contemp. Math., vol.~67, Amer. Math. Soc., Providence, RI, 1987, pp.~25--91.
	
	\bibitem{Drinfeld}
	V.~G. Drinfeld, \emph{Elliptic modules}, Mat. Sb. (N.S.) \textbf{94} (1974),
	594--627.
	
	\bibitem{DrinfeldSym}
	\bysame, \emph{Coverings of {$p$}-adic symmetric domains}, Funkcional. Anal. i
	Prilo\v zen. \textbf{10} (1976), no.~2, 29--40.
	
	\bibitem{DrinfeldES}
	\bysame, \emph{Commutative subrings of certain noncommutative rings},
	Funkcional. Anal. i Prilo\v zen. \textbf{11} (1977), no.~1, 11--14, 96.
	
	\bibitem{GekelerADM}
	E.-U. Gekeler, \emph{Zur {A}rithmetik von {D}rinfeld-{M}oduln}, Math. Ann.
	\textbf{262} (1983), no.~2, 167--182.
	
	\bibitem{GekelerFDM}
	\bysame, \emph{On finite {D}rinfeld modules}, J. Algebra \textbf{141} (1991),
	no.~1, 187--203.
	
	\bibitem{GreenBook}
	E.-U. Gekeler, M.~van~der Put, M.~Reversat, and J.~Van~Geel (eds.),
	\emph{Drinfeld modules, modular schemes and applications}, World Scientific
	Publishing Co., Inc., River Edge, NJ, 1997.
	
	\bibitem{Goss}
	D.~Goss, \emph{Basic structures of function field arithmetic}, Ergebnisse der
	Mathematik und ihrer Grenzgebiete (3), vol.~35, Springer-Verlag, Berlin,
	1996.
	
	\bibitem{HartlIsog}
	U.~Hartl, \emph{Isogenies of abelian anderson {$A$}-modules and {$A$}-motives},
	preprint.
	
	\bibitem{Hausberger}
	T.~Hausberger, \emph{Uniformisation des vari\'et\'es de
		{L}aumon-{R}apoport-{S}tuhler et conjecture de {D}rinfeld-{C}arayol}, Ann.
	Inst. Fourier (Grenoble) \textbf{55} (2005), 1285--1371.
	
	\bibitem{HayesCFT}
	D.~Hayes, \emph{Explicit class field theory in global function fields}, Studies
	in algebra and number theory, Adv. in Math. Suppl. Stud., vol.~6, Academic
	Press, New York-London, 1979, pp.~173--217.
	
	\bibitem{Jordan}
	B.~Jordan, \emph{Points on {S}himura curves rational over number fields}, J.
	Reine Angew. Math. \textbf{371} (1986), 92--114.
	
	\bibitem{JLlocal}
	B.~Jordan and R.~Livn{\'e}, \emph{Local {D}iophantine properties of {S}himura
		curves}, Math. Ann. \textbf{270} (1985), no.~2, 235--248.
	
	\bibitem{Lafforgue}
	L.~Lafforgue, \emph{Chtoucas de {D}rinfeld et conjecture de
		{R}amanujan-{P}etersson}, Ast\'erisque (1997), no.~243, ii+329.
	
	\bibitem{LaumonCDV}
	G.~Laumon, \emph{Cohomology of {D}rinfeld modular varieties. {P}art {I}},
	Cambridge Studies in Advanced Mathematics, vol.~41.
	
	\bibitem{LRS}
	G.~Laumon, M.~Rapoport, and U.~Stuhler, \emph{{$\mathscr{D}$}-elliptic sheaves
		and the {L}anglands correspondence}, Invent. Math. \textbf{113} (1993),
	no.~2, 217--338.
	
	\bibitem{PapikianCrelle1}
	M.~Papikian, \emph{Modular varieties of {$\mathscr{D}$}-elliptic sheaves and
		the {W}eil-{D}eligne bound}, J. Reine Angew. Math. \textbf{626} (2009),
	115--134.
	
	\bibitem{PapikianMZ}
	\bysame, \emph{Endomorphisms of exceptional {$\sD$}-elliptic sheaves}, Math. Z.
	\textbf{266} (2010), no.~2, 407--423.
	
	\bibitem{PapikianPAMS}
	\bysame, \emph{On finite arithmetic simplicial complexes}, Proc. Amer. Math.
	Soc. \textbf{139} (2011), no.~1, 111--124.
	
	\bibitem{PapLocProp}
	\bysame, \emph{Local diophantine properties of modular curves of
		{$\mathscr{D}$}-elliptic sheaves}, J. Reine Angew. Math. \textbf{664} (2012),
	115--140.
	
	\bibitem{Reiner}
	I.~Reiner, \emph{Maximal orders}, London Mathematical Society Monographs. New
	Series, vol.~28, The Clarendon Press, Oxford University Press, Oxford, 2003,
	Corrected reprint of the 1975 original, With a foreword by M. J. Taylor.
	
	\bibitem{RibetBM}
	K.~Ribet, \emph{Bimodules and abelian surfaces}, Algebraic number theory, Adv.
	Stud. Pure Math., vol.~17, Academic Press, Boston, MA, 1989, pp.~359--407.
	
	\bibitem{SerreLF}
	J.-P. Serre, \emph{Local fields}, Graduate Texts in Mathematics, vol.~67.
	
	\bibitem{Shimura}
	G.~Shimura, \emph{On the real points of an arithmetic quotient of a bounded
		symmetric domain}, Math. Ann. \textbf{215} (1975), 135--164.
	
	\bibitem{SilvermanII}
	J.~Silverman, \emph{Advanced topics in the arithmetic of elliptic curves},
	Graduate Texts in Mathematics, vol. 151, Springer-Verlag, New York, 1994.
	
	\bibitem{TaelmanPhD}
	L.~Taelman, \emph{On $t$-motifs}, 2007, Thesis (Ph.D.)--The University of
	Groningen.
	
	\bibitem{vdHeiden}
	G.-J. van~der Heiden, \emph{Weil pairing for {D}rinfeld modules}, Monatsh.
	Math. \textbf{143} (2004), no.~2, 115--143.
	
	\bibitem{Vigneras}
	M.-F. Vign\'eras, \emph{Arithm\'etique des alg\`ebres de quaternions}, Lecture
	Notes in Mathematics, vol. 800, Springer, Berlin, 1980.
	
\end{thebibliography}

\providecommand{\bysame}{\leavevmode\hbox to3em{\hrulefill}\thinspace}

\end{document}